\documentclass[12pt]{article}
\usepackage[dvipsnames,usenames]{color}
\usepackage{amsmath}
\usepackage{amsfonts}
\usepackage{amssymb}
\usepackage{mathrsfs}
\usepackage{amsthm}

\newcommand{\CC}{{\mathbb C}}

\newcommand{\HH}{{\mathbb H}}
\newcommand{\BB}{{\mathbb B}}

\newcommand{\ZZ}{{\mathbb Z}}

\newcommand{\ov}{\overline}
\newcommand{\p}{\partial}
\newcommand{\w}{\widetilde}

\newtheorem{thm}{Theorem}[section]

\newtheorem{lem}[thm]{Lemma}
\newtheorem{prop}[thm]{Proposition}

\numberwithin{equation}{section}

\medskip

\begin{document}
\title{\bf D'Angelo conjecture in the third gap interval}
\date{}
\author{Shanyu Ji\footnote{Supported in part by NSFC-11571260}
\ and Wanke Yin\footnote{Supported in part by NSFC-11722110 and
NSFC-11571260.} } \maketitle

\abstract{ We show the D'Angelo conjecture holds in the third gap
interval. More precisely, we prove that the degree of any rational
proper holomorphic map from $\mathbb{B}^n$ to $\mathbb{B}^{4n-6}$
with $n\geq 7$ is not more than $3$.}

\section{Introduction}

Let $\BB^n=\{z\in\CC^n\ |\ |z|<1\}$ be the unit ball in $\CC^n$ and
denote by $Rat(\BB^n, \BB^N)$ the set of all proper holomorphic
rational maps from  $\BB^n$ to  $\BB^N$. We say that $f, g\in
Rat(\BB^n, \BB^N)$ are {\it holomorphically equivalent} (or {\it
equivalent}, for short) if there are $\sigma\in Aut(\BB^n)$ and
$\tau\in Aut(\BB^N)$ such that $f=\tau\circ g\circ \sigma$. By a
well-known result of Cima-Suffridge \cite{CS}, $F$ extends
holomorphically across the boundary $\p {\BB}^n$.

For the equal dimensional case $N=n$, Alexander \cite{A77} proved
that $Rat(\BB^n, \BB^n)$ must be automorphisms for $n>1$.
Subsequently, much effort has been paid to the classification of
$Rat(\BB^n, \BB^N)$ with $N>n$. When $N / n$ is not too large, it
turns out that the maps are equivalent to relatively simple ones. In
fact, the classification problem had been done for $N\leq 3n-3$ and
the maps turn out to be all monomial maps. The systematic
investigations on the precise classification of $Rat(\BB^n, \BB^N)$
can be found in the work of
\cite{F82,Hu99,HJ01,Ha05,HJX06,CJY18,JY18}, etc. In \cite{FHJZ10},
Faran-Huang-Ji-Zhang constructed a family of  maps in $Rat(\BB^n,
\BB^{3n-2})$, which cannot be equivalent to any polynomial maps.
This indicates 
that the maps could be quite
complicated when $N\geq 3n-2$.

\medspace

To study maps in $Rat(\BB^n, \BB^N)$ , there are two geometric
problems which are of fundamental importance. The first one is the
D'Angelo conjecture.
For any rational holomorphic map $H=\frac{(P_1, ..., P_m)}{Q}$ on
$\CC^n$ where $P_j, Q$ are holomorphic polynomials with $(P_1, ...,
P_m, Q)=1$, the {\it degree} of $H$ is defined, as in algebraic
geometry,  to be $deg(H):=\max\{\deg(P_j), deg(Q),  1\le j\le m\}$.
The D'Angelo conjecture is as follows: For any $F\in Rat(\BB^n,
\BB^N)$, it should have
\[
deg(F) \le \begin{cases}
2N-3,\ & if\ n=2,\\
\frac{N-1}{n-1},\ & if\ n\ge 3.
\end{cases}
\]
There are several partial results supporting this conjecture. The
conjecture is true for all monomial maps, as demonstrated   by
D'Angelo-Kos-Riehl \cite{DKR03} for the case  $n=2$ and
by  Lebl-Peter \cite{LP12} for the case  $n \ge 3$. If $F$ is a
rational map with geometric rank one, this conjecture was proved in
\cite[corollary 1.3]{HJX06}. If the conjecture is proved, it would
be sharp due to the known examples. Also, it is proved that
$deg(F)\le \frac{N(N-1)}{2(2n-3)}$ holds for any $F\in Rat(\BB^n,
\BB^N)$ with $n=2$ in \cite{Me06} and with $n\ge 2$ in \cite{DL09}.

\medskip

Another geometric problem is the gap conjecture.
For any integer $k$ with $1\le k\le K(n)$ where $K(n):=\max\{t\in
\ZZ^+\ |\ \frac{t(t+1)}{2}<n\}$, we recall the {\it gap interval}
(cf. \cite{HJY09})
\begin{equation}
{\cal I}_k:=\bigg(k n,\ (k+1)n - \frac{k(k+1)}{2} \bigg).
\end{equation}
The  {\it gap conjecture}, first raised in \cite{HJY09}, is stated
as follows: Any proper holomorphic rational map $F\in  Rat(\BB^n,
\BB^N)$ where $n\ge 3$ is equivalent to a map of the form $(G, 0')$
where $G\in Rat(\BB^n, \BB^{N'})$ where $N'<N$ if and only if $N\in
{\cal I}_k$ for some $1\le k\le K(n)$. The gap conjecture  for the
cases of $k=1,2,3$ have been proved \cite{Hu99,HJ01,HJY14}.
Recently, P. Ebenfelt \cite{Eb16} proposed a SOS conjecture (i.e.,
the Sums of Squares of Polynomial conjecture)  and proved that if
the SOS conjecture is true, then it implies the gap conjecture.

\medskip

The first gap interval is ${\cal I}_1=(n, 2n-1)$. Huang \cite{Hu99}
showed that $deg(F)=1$ if $N\in {\cal I}_1$.  When $N=2n-1$, it was
proved by Faran \cite{F82} that $deg(F)\le 3$ for $n=2$ and by
Huang-Ji \cite{HJ01}  that $deg(F)\le 2$ for $n>2$. The second gap
interval is ${\cal I}_2=(2n, 3n-3)$. When $N \le 3n-3$ (and $n\ge
4$), we know from \cite{AHJY15} that $deg(F)\le 2$. These results
confirm the D'Angelo conjecture for the first and the second gap
intervals.

\medskip

The third gap interval is ${\cal I}_3=(3n, 4n-6)$. If D'Angelo
conjecture is true, we would have $deg(F)\le 3$ for any $F\in
Rat(\BB^n, \BB^{4n-6})$ because $deg(F)\le
\frac{4n-6-1}{n-1}=4-\frac{3}{n-1}$. This is confirmed by our main
result of this paper as follows.

\medskip

\begin{thm}\label{mainthm}
    If $F\in Rat(\BB^n, \BB^{4n-6})$ with $n \ge 7$, then $deg(F)\le 3$.
\end{thm}

The rest of the paper is organized as follows. In Section 2, we
introduced some known properties for Rat$(\HH_n,\HH_N)$, especially
for maps of geometric rank $2$. Section 3 was devoted to the proof
of our main theorem assuming Proposition \ref{propdeg}. In Sections
4-7, we gave a detailed proof of Proposition \ref{propdeg} according
to four different cases.

\medspace


\medskip

\section{Preliminaries}


Let $\HH_n=\{(z, w)\in\CC^{n-1}\times \CC\ |\ \text{Im}(w)>|z|^2\}$
be the Siegel upper half space and denote by $Rat(\HH_n, \HH_N)$ the
set of all proper holomorphic rational maps from $\HH_n$ to $\HH_N$.
By the Cayley transform, we can identify $\BB^n$ with  $\HH_n$ and
identify $Rat(\BB^n,\BB^N)$ with $Rat(\HH_n, \HH_N)$. In what
follows, we will prove Theorem \ref{mainthm} through the properties
of $Rat(\HH_n, \HH_N)$.

Let $F=(f,\phi,g)=(\widetilde{f}, g)= (f_1,\cdots$, $f_{n-1}$,
$\phi_1,\cdots$, $\phi_{N-n},g)\in Rat(\HH_n, \HH_{N})$. For each
$p\in \p\HH_n$, define $\sigma^0_p\in \hbox{Aut}(\HH_n)$ and
$\tau^F_p\in\hbox{Aut}(\HH_N)$ as follows:
\begin{equation*}\begin{split}
&\sigma^0_p(z,w)=(z+z_0, w+w_0+2i \langle z,\overline{z_0}
\rangle),\\
&\tau^F_p(z^*,w^*)=(z^*-\widetilde{f}(z_0,w_0),w^*-\overline{g(z_0,w_0)}-
2i \langle z^*,\overline{\widetilde{f}(z_0,w_0)} \rangle).
\end{split}\end{equation*}
 Then
$F$ is equivalent to $F_p:=\tau^F_p\circ F\circ
\sigma^0_p=(f_p,\phi_p,g_p)$ and  $F_p(0)=0$. In \cite{Hu99}, Huang
constructed an automorphism $\tau^{**}_p\in
  {Aut}_0({\HH}_N)$ such that
$F_{p}^{**}:=\tau^{**}_p\circ F_p$ satisfies the following
normalization:
\[
f^{**}_{p}=z+{\frac{i}{ 2}}a^{**(1)}_{p}(z)w+o_{wt}(3),\ \phi_p^{**}
={\phi_p^{**}}^{(2)}(z)+o_{wt}(2), \ g^{**}_{p}=w+o_{wt}(4).
\]

Write
$\mathcal{A}(p):=-2i(\frac{\partial^2(f_p)^{\ast\ast}_l}{\partial
z_j
    \partial w}|_0)_{1\leq j,l\leq n-1}$. The {\it geometric rank} of $F$ at
$p$ is defined to be  the rank of the $(n-1)\times (n-1)$ matrix
$\mathcal{A}(p)$, which is denoted by $Rk_F(p)$. Now we define the
{\it geometric rank} of $F$ to be $\kappa_0(F)=max_{p\in
\partial\HH_n} Rk_F(p)$. \medspace

When a map in $ Rat({\HH}_n,{\HH}_N)$ is not of full rank (i.e.,
$\kappa_0\le n-2$),  by the works of \cite{Hu03} and \cite{HJX06},
it can further be normalized to the following form:
\begin{thm}
    \label{normalize **}
    Suppose that $F\in Rat({\HH}_n,{\HH}_N)$
    has geometric rank  $1\le\kappa_0\le n-2$ with $F(0)=0$. Then there are
    $\sigma\in \hbox{Aut}({\HH}_n)$ and
    $\tau\in \hbox{Aut}({\HH}_N)$ such that
    $\tau\circ F\circ \sigma $ takes
    the following form, which is still denoted by $F=(f,\phi,g)$ for
    convenience of notation:

    \begin{equation}
    \left\{
    \begin{array}{l}
    f_l=\sum_{j=1}^{\kappa_0}z_jf_{lj}^*(z,w),\ l\le\kappa_0,\\
    f_j=z_j,\  \text{for} \ \kappa_0+1\leq j\leq n-1,\\
    \phi_{lk}=\mu_{lk}z_lz_k+\sum_{j=1}^{\kappa_0}z_j\phi^*_{lkj}\ \text{for
    } \ \ (l,k)\in {\cal S}_0,\\
    \phi_{lk}=O_{wt}(3),\ \ (l,k)\in {\cal S}_1,\\
    g=w;\\
    f_{lj}^*(z,w)=\delta_l^j+\frac{i\delta_{l}^j\mu_l}{2}w+b_{lj}^{(1)}(z)w+O_{wt}(4),\\
    \phi^*_{lkj}(z,w)=O_{wt}(2),\ \  (l,k)\in {\cal S}_0,\\
    \phi_{lk}=\sum_{j=1}^{\kappa_0}z_j\phi_{lkj}^*=O_{wt}(3)\ \ for\
    (l,k)\in {\cal S}_1.
    \end{array}\right.
    \label{eqn:hao}
    \end{equation}
    Here, for $1\le \kappa_0\le n-2$, we write ${\cal S} ={\cal S}_0\cup
    {\cal S}_1$, the index set  for all components of $\phi$, where
    ${\cal S}_{0}=\{(j,l): 1\le j\leq \kappa_0, 1\leq l\leq n-1, j\leq
    l\}$, $ {\cal S}_1=\Big\{(j, l): j=\kappa_0+1, \kappa_0+1\le l \le \kappa_0 +
    N-n-\frac{(2n-\kappa_0-1)\kappa_0}{2} \Big\}$, and
    \begin{equation}
    \label{mui and mujk} \mu_{jl}=\begin{cases}\sqrt{\mu_j+\mu_l} &\
    for\ j<l\le
    \kappa_0; \\
    \sqrt{\mu_j} & if \ j\le \kappa_0 < l\ or\ if\ j=l\le \kappa_0.
    \end{cases}
    \end{equation}
\end{thm}

Write $\phi=(\Phi_0, \Phi_1)$ where $\Phi_0=(\phi_{lk})_{(l,k)\in
{\cal S}_0}$, $\Phi_1=(\phi_{lj})_{(l,k)\in{\cal S}_1}$. Define
$\Phi_0^{(1,1)}(z)=\sum^{\kappa_0}_{j=1} e_j z_j$,
$\Phi_1^{(1,1)}(z)=\sum^{\kappa_0}_{j=1} \hat{e_k}z_j$, $e^*_j=(e_j,
\hat e_j)$, and $\xi_j(z)=\ov{e_j}\cdot \Phi^{(2,0)}_0(z)$. Here we
use notation $\Phi^{(s,t)}(z)$ to denote a homogeneous polynomial of
degree $s$ in $z$, and
\begin{equation*}
\begin{split}
&\Phi^{(s,t)}(z, w)=\Phi^{(s,t)}(z) w^t,\
\Phi^{(s,t)}(z)=\sum\limits_{s_1+\cdots+s_{n-1}=s}
\Phi^{(s_1I_1+\cdots+s_{n-1}I_{n-1}+tI_n)}z_1^{s_1}\cdots
z_{n-1}^{s_{n-1}}.
\end{split}
\end{equation*}

Let  $F$ be as in Theorem \ref{normalize **}. By \cite[Corollary
3.4]{HJY14}, we have the following explicit expression for
$\Phi^{(3,0)}_1(z)$.
\begin{lem}
    \label{phi30} Let $\kappa_0\ge 2$ and $(\kappa_0+1)n-\kappa_0
    \le N\le (\kappa_0+2)n - \kappa_0^2 -2$. Then, after applying
    a unitary transformation to the $\Phi_1$-components,
    $F$ is equivalent to another map such that the new map
    (still denoted as $F$) has the property
    \begin{equation}
    \label{Phi(3,0)1}
    \Phi^{(3,0)}_1(z)=\bigg(\frac{2}{\sqrt{\mu_j+\mu_l}}
    \bigg[\sqrt{\frac{\mu_j}{\mu_l}}z_j \xi_l -
    \sqrt{\frac{\mu_l}{\mu_j}} z_l \xi_j\bigg],\ \ 0'  \bigg)_{1\le
        j<l\le \kappa_0}.
    \end{equation}
\end{lem}
We call a map satisfying Theorem \ref{normalize **} and
(\ref{Phi(3,0)1}) a {\it normalized map}, and denote it by
$F^{***}$.

\medspace

From the above lemma and \cite[(3.7)]{HJY14}, we obtain

\begin{equation}
\label{phi (3,0)} \phi_{jk}^{(3,0)}(z)=\begin{cases}
\frac{2i}{\mu_{jk}}[z_j \xi_k(z) + z_k \xi_j(z)],\ \ if\ j<k\le \kappa_0,\\
\frac{2i}{\mu_{jk}} z_l \xi_j(z),\ \ if\ j=k,\ or\ j\le \kappa_0 <k,
\end{cases}
\end{equation}
where $\xi=(\xi_1, ..., \xi_{\kappa_0})=\ov{(e_1, ...,
e_{\kappa_0})} \cdot \Phi_0^{(2,0)}$.
\medskip

Next, we show more properties for $F\in Rat({\HH}_n,{\HH}_{4n-6})$
with $n\geq 7$ and geometric rank $\kappa_0=2$.

From (\ref{Phi(3,0)1}), we have
\begin{equation}
\label{phi 1}
\begin{cases}
\phi_{33}^{(3I_1)} = \frac{2\mu_1}{\sqrt{\mu_2(\mu_1+\mu_2)}}
\ov{e_{2,11}},\ \ \phi_{33}^{(3I_2)} = - \frac{2
\mu_2}{\sqrt{\mu_1(\mu_1+\mu_2)}} \ov{e_{1,22}}, \ \
\phi_{33}^{(3I_k)} = 0,\\
\phi_{33}^{(2I_1+I_2)} = \frac{2\sqrt{\mu_1}}{\sqrt{\mu_2}}
\ov{e_{2,12}} - \frac{2\sqrt{\mu_2}}{\sqrt{\mu_1+\mu_2}}
\ov{e_{1,11}}, \ \
\phi_{33}^{(2I_1+I_j)}=\frac{2 \mu_1}{\sqrt{\mu_2(\mu_1+\mu_2)}} \ov{e_{2,1j}},\\
\phi_{33}^{(I_1+I_2+I_j)} = \frac{2\sqrt{\mu_1}}{\sqrt{\mu_1+\mu_2}}
\ov{e_{2,2j}} - \frac{2\sqrt{\mu_2}}{\sqrt{\mu_1+\mu_2}} \ov{e_{1,1j}},\\
\phi_{33}^{(I_1+2I_2)} = \frac{2\sqrt{\mu_1}}{\sqrt{\mu_1+\mu_2}}
\ov{e_{2,22}} - \frac{2\sqrt{\mu_2}}{\sqrt{\mu_1}} \ov{e_{1,12}}, \
\
\phi_{33}^{(2I_2+I_j)} = - \frac{2 \mu_2}{\sqrt{\mu_1(\mu_1+\mu_2)}} \ov{e_{1,2j}},\\
\phi_{33}^{(I_1+2I_k)} = \phi_{33}^{(I_2+2I_k)} =
\phi_{33}^{(I_1+I_k + I_j)} = \phi_{33}^{(I_2+I_k+I_j)} =
\phi_{33}^{(I_i+I_k+I_l)} = 0.
\end{cases}
\end{equation}

From \cite[(4.20) and (4.46)]{HJY14}, we know
\begin{equation}
\label{phi 2}
\begin{split}
& \phi_{11}^{(2,1)}(z) - 2i \sum^2_{j=1} \frac{\xi_j}{\mu_j}
e_{j,11}
= \frac{-2i}{\sqrt{\mu_1}}z_1 f_1^{(1,2)}(z),\\
&\phi_{12}^{(2,1)}(z) - 2i \sum^2_{j=1} \frac{\xi_j}{\mu_j} e_{j,12}
= \frac{-2i}{\sqrt{\mu_1+\mu_2}}(z_1 f_2^{(1,2)}(z)+z_2 f_1^{(1,2)}(z)),\\
& \phi_{22}^{(2,1)}(z) - 2i \sum^2_{j=1} \frac{\xi_j}{\mu_j}
e_{j,22}
= \frac{-2i}{\sqrt{\mu_2}}z_2 f_2^{(1,2)}(z),\\
& \phi_{1\alpha}^{(2,1)}(z) - 2i \sum^2_{j=1} \frac{\xi_j}{\mu_j}
e_{j,1\alpha}
= \frac{-2i}{\sqrt{\mu_1}}z_\alpha f_1^{(1,2)}(z),\\
& \phi_{2\alpha}^{(2,1)}(z) - 2i \sum^2_{j=1} \frac{\xi_j}{\mu_j}
e_{j,2\alpha}
= \frac{-2i}{\sqrt{\mu_2}}z_\alpha f_2^{(1,2)}(z),\\
& \phi_{33}^{(2,1)}(z) - 2i \sum^2_{j=1} \frac{\xi_j}{\mu_j}
e_{j,33} =
\frac{-2}{\sqrt{\mu_1+\mu_2}}\bigg(\sqrt{\frac{\mu_1}{\mu_2}} z_1
f_2^{(1,2)}(z)
-\sqrt{\frac{\mu_2}{\mu_1}} z_2 f_1^{(1,2)}(z)\bigg).\\
\end{split}
\end{equation}

\medspace

By \cite[(4.3)]{HJY14}, we have $\frac{i}{2} \mu_1
a_n^{(1)}(\epsilon)+f_1^{(1,2)}(\epsilon,0,...,0)=0,
\frac{i}{2}\mu_2 b_n^{(1)}(\epsilon)+f_2^{(1,2)}(\epsilon,0,...,0)$
$=0$, and $\phi^{(1,2)}(\epsilon,0,...,0)+e_1^* a_n^{(1)}(\epsilon)
+ e_2^* b_n^{(1)}(\epsilon)=0$. Thus
\begin{equation}
\label{phi 3} \phi^{(1,2)}(\chi, 0)=-\frac{2i}{\mu_1}
f_1^{(1,2)}(\chi, 0) e_1^* - \frac{2i}{\mu_2} f_2^{(1,2)}(\chi, 0)
e_2^*.
\end{equation}

\medspace

Next we recall the definitions and the properties of the degeneracy
ranks. These invariants are introduced by Lamel \cite{La01} and
Ebenfelt-Huang-Zaitsev \cite{EHZ04}.

Let $F\in Rat(\HH_n,\HH_N)$ and $p\in \p\HH_n$. Write $L^j=\frac{\p
}{\p z_j} + 2i \ov{z_j} \frac{\p}{\p w}$. Then $\{L^j\}_{1\le j\le
n-1}$ forms a basis of tangent vector fields of $(1,0)$ on
$\p\HH_n$. Denote by $\hat \rho(Z, \ov Z)$  the defining function of
the real hypersurface $\p\HH_N$, and denote by $\hat \rho_{\ov
Z}:=\ov\p \hat \rho$ the complex gradient of $\rho$. Now we define
an increasing sequence of linear subspaces $E_{k}(p) \subset
\CC^{N}$ as follows:
\begin{equation}
E_{k}(p)=span_{\CC}\{L^{\alpha_1} ... L^{\alpha_{n-1}} \hat\rho_{\ov
Z} \circ F(p)\ |\ \ |\alpha_1|+...+|\alpha_{n-1}|\le k\}
\end{equation}
Define $d_1(p):=0$ and $ \label{dk(p)} d_k(p):= \dim_\CC \
E_k(p)/E_1(p). $  By moving $p$ to a nearby point $p_0$ if
necessary, we may assume that all $d_l(p)$ are locally constant near
$p_0$ and
\begin{equation}
\label{rank} d_2(p)<d_3(p)<...<d_{l_0}(p)=d_{l_0+1}(p)=....
\end{equation}
for some $l_0$ with $1\le l_0\le N-n+1$.  By [EHZ04], $l_0$  is
called the {\it degeneracy rank} of $F$, and  $d_{l_0}$ is called
the {\it degeneracy dimension} of $F$. These definitions depend on
the open subset $U$.  {\it degeneracy rank} $l_0$ of $F$ is defined
to be the
 minimization of $l_0$ among all these open sets. Then we have the
 following degree estimates when both $\kappa_0$ and $l_0$ are
 small.

\begin{thm} \label{nondegenerate rank}
    (\cite[Theorem 1.4]{CJY18})\ \ \
    Let $F\in Rat(\BB^n, \BB^N)$ with $5\le n\le N$. Suppose that  $\kappa_0\le 2$ and $l_0\le 2$. Then $deg(F)\le 2$.
\end{thm}

Finally, we finish this section by recalling a lemma of \cite{HJ01},
which gives a reduction for the degree estimates of a rational
holomorphic map.

 For any point $q=(\w q, q_n)\in \CC^n$, the Segre family of
$\p\HH_n$ is a family given by $Q_q:=\{(\zeta, \tau)\in
\CC^{n-1}\times \CC\ |\ \tau=\ov{q_n}+ 2i \langle \zeta, \w
q\rangle) \}$. Let $F:\CC^n\to \CC^N$ be a rational holomorphic map.
Then the restriction $F|_{Q_q}(\zeta)= F(\zeta, \ov{q_n}+2i \langle
\zeta, \w q\rangle)$ is a rational holomorphic map in $\zeta\in
\CC^{n-1}$.

\begin{lem}
\label{degree} (\cite[lemmas 5.3-5.4]{HJ01}) Let $F=\frac{(P_1, ...,
P_N)}{Q}$ be a rational holomorphic map from $\CC^n$ into $\CC^N$
where $P_j$ and $Q$ are as above. Suppose that there exists a fixed
positive integer $k$ such that one of the following two conditions
is satisfied:
\begin{enumerate}
\item $deg(F|_{Q_q})\le k$ for all $q\in \p\HH_n$.
\item $deg(F^{**}_q|_{Q_0})\le k$ for all $q\in \p\HH_n$.
\end{enumerate}
Then $deg(F)\le k$.
\end{lem}

\section{Proof of Theorem \ref{mainthm}}

This section is devoted to the proof of Theorem \ref{mainthm},
assuming Proposition \ref{propdeg}, whose proof will be given in the
end of this section and in $\S 4-7$.

We start with  some notations. For any $F\in Rat(\HH_n, \HH_{4n-6})$
and $q\in \p \HH_n$, the map associates to a normalized map
$F_q^{***}$. We write $F_q^{***}=(f^{***}_q, \phi^{***}_q,
g^{***}_q)$ and $
\phi_{q}^{***}=\big((\Phi_{q}^{***})_0,(\Phi_{q}^{***})_1\big) $,
which are decomposed similar to those of $F$ in Theorem
\ref{normalize
**}.

\begin{lem}
\label{lemma N=3n-2} Let $F\in Rat(\HH_n, \HH_{4n-6})$ be a
normalized map with $\kappa_0=2$ and $n\ge 7$. Then
\begin{equation}
\label{lemma 2.5 1} D^\alpha F(0)\in span_{|\beta|\le 3} \{D^\beta
F(0)\}
\end{equation}
for any index $\alpha$,  where $D^\alpha=\frac{\p^{|\alpha|}}{\p
z_1^{\alpha_1}\cdots \p z_n^{\alpha_n}}$ is the standard
differential operator. If, in addition,
$\big(\Phi^{***}_{q}\big)_1^{(3,0)}(z)\equiv 0$ for any
 $q$  in
an open neighborhood of $0$ in $\p\HH_n$, then for any index
$\alpha$,
\begin{equation}
\label{lemma 2.5 2} D^\alpha F(0)\in span_{|\beta|\le 2} \{D^\beta
F(0)\}.
\end{equation}
\end{lem}

\medspace

\begin{proof} By \cite[Theorem 4.1 and (5.3)]{HJY14}, we have
$span_{|\beta|\le 4}\{ {\cal L}^\beta F|_0\} \le span_{|\alpha|\le
3}\{ {\cal L}^\alpha F|_0\}$. By \cite[claim (5.5)]{HJY14}, for any
$p\in\p\HH_n$, the associated map $F_p$ satisfies
\[
span_{|\beta|\le 4}\{ {\cal L}^\beta F_p|_0\} \le span_{|\alpha|\le
3}\{ {\cal L}^\alpha F_p|_0\}.
\]
Since ${\cal L}^\alpha (F_p)|_0={\cal L}^\alpha F(p)$, it implies
$span_{|\beta|\le 4}\{ {\cal L}^\beta F|_p\} \le span_{|\alpha|\le
3}\{ {\cal L}^\alpha F|_p\}$. Then  by similar argument as that in
\cite[p. 139]{HJY14}, through  applying $\ov{\cal L}_j, {\cal L}_k$
to $F$, we obtain (\ref{lemma 2.5 1}). If, in addition,
{$(\Phi^{***}_q)_1^{(3,0)}(z)\equiv 0$ for any
 $q$  in
an open neighborhood of $0$ in $\p\HH_n$}, by the last paragraph of
\cite[p. 139]{HJY14}, we get (\ref{lemma 2.5 2}).
\end{proof}


\begin{proof}[Proof of Theorem \ref{mainthm}]

As mentioned in $\S 2$, by the Cayley transform, we can
 identify $Rat(\BB^n,\BB^N)$ with $Rat(\HH_n, \HH_N)$. It suffices for us to show for
  any $F\in Rat(\HH_n, \HH_{4n-6})$ with $n\geq 7$,
  we have deg$(F)\leq 3$.

  By \cite[Theorem 1.1]{Hu03}, we have $N \ge n
  +\frac{(2n-\kappa_0-1)\kappa_0}{2}$. Hence for $N=4n-6$, its
 geometric rank $\kappa_0\le 3$. By \cite{Hu99}, we can suppose $1\leq \kappa_0\le 3$, otherwise the map must be linear fractional.
  If  $F\in Rat(\HH_n,
\HH_{4n-6})$ with $n \ge 3$ and $\kappa_0=1$, by \cite[Corollary
1.3]{HJX06}, we have $deg(F)\le 3$.  If $F\in Rat(\HH_n,
\HH_{4n-6})$ with $n \ge 5$ and $\kappa_0=3$, by  \cite[ Theorem
1.1]{JY18}, we have $deg(F)\le 2$. Therefore, we assume $\kappa_0=2$
for $F$ in the rest proof of Theorem \ref{mainthm}.

Let $F\in Rat(\HH_n, \HH_{4n-6})$ with $n \ge 7$ and $\kappa_0=2$.
For any $q\in \p\HH_n$, we study the associated map $(F_{q})^{***}$.
There are two cases to consider.

\medspace

{\bf Case I:} $F^{***}_{q}$ satisfies
$\big(\Phi_{q}^{***}\big)_1^{(3,0)}(z)\equiv 0$ for any $q$ in a
neighborhood of $0$ in $\p\HH_n$. Recall that $F^{***}_{q}$ is a
normalized map with geometric rank $2$. Hence by (\ref{lemma 2.5
2}), it implies that the degeneracy rank $l_0\le 2$. Then we apply
Theorem \ref{nondegenerate rank} to conclude $deg(F)\le 2$.

\medspace

{\bf Case 2:} $F^{***}_{q_0}$ satisfies that
$\big(\Phi_{q_0}^{***}\big)_1^{(3,0)}(z)\not\equiv 0$  for some
$q_0\in \p\HH_n$. As in Case I,  $F^{***}_{q_0}$ is a normalized map
with geometric rank $2$.  Now we need the following proposition.

\begin{prop}\label{propdeg}
  Let $F\in Rat(\HH_n,\HH_{4n-6})\ (n\geq 7)$ be a normalized map, whose geometric rank
  $\kappa_0=2$. Suppose that $(\Phi_1)^{(3,0)}(z)\not\equiv 0$.Then
  \begin{equation}\label{reduction 0}
 deg(f(z, 0), \phi(z, 0))\le 3.
\end{equation}
\end{prop}

The proof of Proposition \ref{propdeg} will be postponed to the next
section. Admitting this proposition temporarily, we continue the
proof of Theorem \ref{mainthm}.

Notice that $(\Phi_{q}^{***})_1^{(3,0)}(z)$ is smooth with respect
to $q$. From $(\Phi_{q_0}^{***})_1^{(3,0)}(z)\not\equiv 0$, we know
$(\Phi_{q}^{***})_1^{(3,0)}(z)\not\equiv 0$ for $q$ in a small
neighborhood of $q_0$. Applying Proposition \ref{propdeg} with $F$
replaced by $F_q^{***}$ for any $q\in \p\HH_n$ in a neighborhood of
$q_0$, we conclude  that deg$(F_q^{***}(z, 0))\le 3$. Since maps in
both Aut$(\p \HH_n) $ and Aut$(\p \HH_N) $ are linear fractional, we
know deg$(F_q(z, 0))\le 3$. Now we use Lemma \ref{degree}  to
complete the proof of Theorem \ref{mainthm}.
\end{proof}

Next we proceed to the proof of Proposition \ref{propdeg}.

\begin{proof}[Proof of Proposition \ref{propdeg}]

By \cite[Lemma 2.3]{HJX06}, we know $f_l(z, 0)=z_l$. From Lemma
\ref{phi30}, $\Phi_1=(\phi_{j,k})_{(j,k)\in {\cal S}_1} =(\phi_{33},
\phi_{34}, ..., \phi_{3 (n-2)})$ satisfies $\Phi^{(3)}_1(z,
0)=(\phi^{(3)}_{33}(z, 0), 0, ..., 0)$. Then we apply Lemma
\ref{lemma N=3n-2} to know
\begin{equation}
\label{reduction 2} \phi_{34}(z, 0)\equiv 0, \cdots, \phi_{3
(n-2)}(z, 0)\equiv 0.
\end{equation}
Hence  it suffices to prove
\begin{equation}
\label{reduction} deg\bigg(deg(\phi_{jk}(z, 0)), \ deg(\phi_{33}(z,
0)\bigg) \le 3,\ \ \ \forall (j,k)\in {\cal S}_0.
\end{equation}

By our notation,
$\phi^{(3,0)}_{33}=(\phi_{33}^{(j_1I_1+j_{n-1}I_{n-1})})_{j_1+\cdots+j_{n-1}=3,j_1+j_2\geq
1 }$. From (\ref{phi 1}) and (\ref{reduction 2}),
$(\Phi_1)_{33}^{(3,0)}(z)\not\equiv 0$ if and only if
\begin{equation}
\label{cases} (\phi_{33}^{(3I_1)}, \phi_{33}^{(3I_2)},
\phi_{33}^{(2I_1+I_2)}, \phi_{33}^{(I_1+2I_2)},
\phi_{33}^{(I_1+I_2+I_j)}, \phi_{33}^{(2I_1+I_j)},
\phi_{33}^{(2I_2+I_j)})\not=(0, 0, ...., 0).
\end{equation}
Thus it is suffices for us to consider the following cases:
\begin{equation}\begin{split}\label{ABCD}
 &\text{Case A: } \phi_{33}^{(3I_1)}\not=0,\ \text{Case A}':\
 \phi_{33}^{(3I_2)}\not=0,\\
&\text{Case B:}\  \phi_{33}^{(2I_1+I_2)}\not=0, \ \text{Case B}': \
\phi_{33}^{(I_1+2I_2)}\not=0,\\
&\text{Case C:}\ \phi_{33}^{(I_1+I_2+I_j)}\not=0,\\
&\text{Case D:} \ \phi_{33}^{(2I_1+I_j)}\not=0, \ \text{Case D}':\
\phi_{33}^{(2I_2+I_j)}\not=0.
\end{split}\end{equation}

In the rest of the paper, we will obtain an explicit expression for
$\phi_{33}(z, 0)$ and  $\phi_{jk}(z, 0)$ with $(j,k)\in {\cal S}_0$,
from which we obtain (\ref{reduction}).

We start with the following Chern-Moser equation.
\[
\frac{g(z, w)-\ov{g(z, w)}}{2i} = f(z, w)\cdot \ov{f(z, w)}+\phi(z,
w) \cdot \ov{\phi(z, w)}, \ \ \ \forall\ \text{Im}(w)=|z|^2.
\]
By complexification, we write
\begin{equation}
\label{basic 1} \frac{g(z, w) - \ov{g(\ov\chi, \ov\eta)}}{2i} =
\sum^{n-1}_{l=1} f_l(z, w) \ov{f_l(\ov\chi, \ov\eta)} +
\sum_{(s,t)\in {\cal S}} \phi_{st}(z, w) \ov{\phi_{st}(\ov\chi,
\ov\eta)},\ \ \forall\ \frac{w-\eta}{2i}=z\cdot \chi.
\end{equation}
Applying ${\cal L}_j:=\frac{\p }{\p z_j} + 2i \chi_j \frac{\p}{\p
    w}$ for $z=0$ and $w=\eta=0$ to the both sides of the above
identity, we obtain
\begin{equation}
\label{basic 2} \frac{{\cal L}_j g(0, 0)}{2i} = \sum^{n-1}_{l=1}
{\cal L}_j f_l(0, 0) \ov{f_l(\ov\chi, 0)} + \sum_{(s, t)\in {\cal
S}_0} {\cal L}_j\phi_{st}(0, 0) \ov{\phi_{st}(\ov\chi, 0)},
\end{equation}

\begin{equation}
\label{basic 3} \frac{{\cal L}_j {\cal L}_k g(0, 0)}{2i} =
\sum^{n-1}_{l=1} {\cal
    L}_j {\cal L}_k f_l(0, 0) \ov{f_l(\ov\chi, 0)} + \sum_{(s, t)\in {\cal S}_0}
{\cal L}_j{\cal L}_k \phi_{st}(0, 0) \ov{\phi_{st}(\ov\chi, 0)}.
\end{equation}
We notice that the index set ${\cal S}$ in (\ref{basic 1}) is
replaced by ${\cal S}_0$ in (\ref{basic 2}) and (\ref{basic 3})
because $\phi_{st}^{(2)}(0,0)=0$ for any $(s, t)\in {\cal S}_1$. The
proof of Proposition \ref{propdeg} will be completed in Sections
$4-7$ according to the different cases in (\ref{ABCD}).
\end{proof}

\section{Proof of Proposition \ref{propdeg} for Case A}

In this section,  we'll proceed to complete the proof of Proposition
\ref{propdeg} for Case A.

\begin{proof}[Proof of Proposition \ref{propdeg} for Case A]
Applying ${\cal L}_1^3$ to (\ref{basic 1}), we obtain
\begin{equation}
\label{basic 4} \frac{{\cal L}_1^3  g(0, 0)}{2i} = \sum^{n-1}_{l=1}
{\cal L}^3_1 f_l(0, 0) \ov{f_l(\ov\chi, 0)} + \sum_{(s, t)\in {\cal
S}_0} {\cal L}_1^3 \phi_{st}(0, 0) \ov{\phi_{st}(\ov\chi, 0)} +
{\cal L}_1^3 \phi_{33}(0, 0) \ov{\phi_{33}(\ov\chi, 0)}.
\end{equation}
Here we have used (\ref{reduction 2}). We can write these equations
in terms of matrix,

\begin{equation}\begin{split}\label{matrixb}
B
\begin{pmatrix}
\ov{f_1(\ov\chi, 0)}\\ \ov{f_{2}(\ov\chi, 0)}\\
\ov{\phi(\ov\chi, 0)}
\end{pmatrix} = \begin{pmatrix}
\chi_1\\
\chi_{2}\\
0\\
\vdots\\
0\\
\end{pmatrix}
\end{split}\end{equation}
where  $B$ is a non-singular $(2n-1)\times (2n-1)$ matrix evaluated
at $(0,0)$:

\begin{equation}\label{matrix}
B:=\begin{pmatrix} {\cal L}_jf_{h} & {\cal L}_j \phi_{hl} &
{\cal L}_j \phi_{h\alpha} & {\cal L}_j \phi_{33}\\
{\cal L}_j{\cal L}_kf_{h} &  {\cal L}_j{\cal L}_k \phi_{hl}
& {\cal L}_j{\cal L}_k \phi_{h\alpha} & {\cal L}_j{\cal L}_k \phi_{33} \\
{\cal L}_j{\cal L}_\beta f_{h} &
{\cal L}_j{\cal L}_\beta \phi_{hl} & {\cal L}_j{\cal L}_\beta \phi_{h\alpha} & {\cal L}_j{\cal L}_\beta \phi_{33}\\
{\cal L}^3_1 f_{h}  & {\cal L}^3_1 \phi_{hl} &
{\cal L}^3_1 \phi_{h\alpha} & {\cal L}^3_1 \phi_{33} \\
\end{pmatrix}_{1\leq j,h,l\leq 2, {3\leq \alpha,\beta\leq n-1}}
\end{equation}
In the rest of this section, we'll use (\ref{matrixb}) to solve
$\phi(z, 0)$.

\medskip

By our normalization (\ref{eqn:hao}), we have
\begin{equation} \label{L2 eq}
\begin{split}
& {\cal L}_1^2 f_1|_{(0,0)}=4i \chi_1 f_1^{(I_1+I_n)} = - 2 \mu_1 \chi_1,\ \ {\cal L}_1^2 f_2|_{(0,0)}=0, \\
& {\cal L}_1{\cal L}_2 f_1|_{(0,0)} = 2i \chi_2 f_1^{(I_1+I_n)} = -
\mu_1 \chi_2,\ \
{\cal L}_1 {\cal L}_2 f_2|_{(0,0)} = 2i \chi_1 f_2^{(I_2+I_n)} = -\mu_2 \chi_1, \\
& {\cal L}_2^2 f_1|_{(0,0)} = 0,\ \ \ {\cal L}_2^2 f_2|_{(0,0)} = 4i \chi_2 f_2^{(I_2+I_n)} = - 2 \mu_2 \chi_2, \\
& {\cal L}_1 {\cal L}_\alpha f_1|_{(0,0)} = 2i \chi_\alpha
f_1^{(I_1+I_n)} = - \mu_1 \chi_\alpha,
\ \ {\cal L}_1 {\cal L}_\alpha f_2|_{(0,0)} = 0,\\
& {\cal L}_2 {\cal L}_\alpha f_1|_{(0,0)} = 0,\ {\cal L}_2 {\cal
L}_\alpha f_2|_{(0,0)} = 2i \chi_\alpha f_2^{(I_2+I_n)} = - \mu_2
\chi_\alpha.
\end{split}
\end{equation}
Notice that $f_1^{(2I_1+I_n)}=-\sqrt{\mu_1}\ov{e_{1,11}},
f_2^{(2I_2+I_n)}= - \sqrt{\mu_1}\ov{e_{2,11}}$. Recall ${\cal
L}_1^3=\frac{\p^3}{\p z_1^3} + 6i \chi_1 \frac{\p^3}{\p z_1^2 \p w}
- 12\chi_1^2 \frac{\p^3}{\p z_1 \p w^2} - 8i \chi_1^3 \frac{\p^3}{\p
w^3}$.
Then
\begin{equation}
\begin{split}
{\cal L}_1^3 f_1|_{(0,0)} & = 6i \chi_1 \cdot 2 f_1^{(2I_1+I_n)}
-12\chi_1^2 \cdot 2 f_1^{(I_1+2I_n)}\\
& = - 12i \chi_1 \cdot \sqrt{\mu_1} \ov{e_{1,11}}-24\chi_1^2 f_1^{(I_1+2I_n)}:=C_1, \\
{\cal L}_1^3 f_2|_{(0,0)} & =  6i \chi_1 \cdot 2 f_2^{(2I_1+I_n)}
-12\chi_1^2 \cdot 2 f_2^{(I_1+2I_n)}\\
& = - 12i \chi_1 \cdot \sqrt{\mu_1} \ov{e_{2,11}} -24\chi_1^2
f_2^{(I_1+2I_n)}:
= C_2. \\
\end{split}
\end{equation}

\medspace

Set
\begin{equation}
B_1:= \begin{pmatrix}
{\cal L}_j{\cal L}_k \phi_{hl} & {\cal L}_j{\cal L}_k \phi_{h\alpha} & {\cal L}_j{\cal L}_k \phi_{33} \\
{\cal L}_j{\cal L}_\beta \phi_{hl} & {\cal L}_j{\cal L}_\beta
\phi_{h\alpha} & {\cal L}_j{\cal L}_\beta
\phi_{33}\\
{\cal L}^3_1 \phi_{hl} & {\cal L}^3_1 \phi_{h\alpha} &
{\cal L}^3_1 \phi_{33} \\
\end{pmatrix}_{1\leq j,h,l\leq 2, {3\leq \alpha,\beta\leq n-1}} \bigg|_{(0,0)}
=\begin{pmatrix}
B_{1, jk}\\
B_{1, j \beta}\\
B_{1, 33}
\end{pmatrix}.
\end{equation}
Then (\ref{matrixb}) takes the following form:
\begin{equation}
\label{new form} B_1 \ov{\phi}(\chi, 0)^t=A_1,
\end{equation}
where $A_1=(A_{1,11}, A_{1,12}, A_{1,22}, A_{1,1\alpha},
A_{1,2\alpha}, A_{1,33})^t$:
\begin{equation}
\begin{split}
& A_{1,11}=-{\cal L}_1^2 f_1|_{(0,0)}\chi_1 - {\cal L}_1^2 f_2|_{(0,0)} \chi_2 = 2 \mu_1 \chi_1^2,\\
& A_{1,12}=-{\cal L}_1{\cal L}_2 f_1 |_{(0,0)}\chi_1 - {\cal
L}_1{\cal L}_2 f_2|_{(0,0)} \chi_2
= \mu_1 \chi_1 \chi_2 + \mu_2 \chi_1\chi_2,\\
\end{split}
\end{equation}
\begin{equation}
\begin{split}
& A_{1,22}=-{\cal L}_2^2 f_1|_{(0,0)} \chi_1 - {\cal L}_2^2 f_2|_{(0,0)} \chi_2 = 2\mu_2 \chi_2^2, \\
& A_{1,1\alpha} = - {\cal L}_1{\cal L}_\alpha f_1|_{(0,0)} \chi_1 - {\cal L}_1{\cal L}_\alpha f_2|_{(0,0)} \chi_2 = \mu_1  \chi_1 \chi_\alpha,\\
& A_{1,2\alpha} = -{\cal L}_1{\cal L}_\alpha f_1|_{(0,0)} \chi_1- {\cal L}_1{\cal L}_\alpha f_2|_{(0,0)}\chi_2 =\mu_2 \chi_2 \chi_\alpha,\\
& A_{1,33}=-{\cal L}_1^3 f_1|_{(0,0)} \chi_1 - {\cal L}_1^3
f_2|_{(0,0)} \chi_2 = -\chi_1 C_1 - \chi_2 C_2,
\end{split}
\end{equation}
and
\begin{equation}
\label{def of B1}
\begin{split}
B_{1,11}&=(2\mu_{11} + 4i \chi_1 e_{1,11}, 4i \chi_1 e_{1,12}, 4i\chi_1 e_{1,22}, 4i\chi_1 e_{1,1\alpha}, 4i \chi_1 e_{1,2\alpha}, 4i \chi_1 e_{1,33}),\\
B_{1,12}&=(2 \chi_1 e_{2,11} + 2i \chi_2 e_{1,11}, \mu_{12} + 2i \chi_1 e_{2,12} + 2i \chi_2 e_{1,12}, 2i\chi_1 e_{2,22}+2i\chi_2 e_{1,22}, \\
&\ \ \ \ \ \ 2i\chi_1 e_{2,1\alpha}+2i\chi_2 e_{1,1\alpha}, 2i
\chi_1 e_{2,2\alpha}
+2i \chi_2 e_{1,2\alpha}, 2i \chi_1 e_{2,33}+2i\chi_2e_{1,33}),\\
B_{1,22}&=(4 i \chi_2 e_{2,11}, 4i \chi_2 e_{2,12}, 2 \mu_{22} + 4i\chi_2 e_{2,22}, 4i\chi_2 e_{2,1\alpha}, 4i \chi_2 e_{2,2\alpha}, 4i \chi_2 e_{2,33}),\\
B_{1,1\beta}&=(2i \chi_\beta e_{1,11}, 2i \chi_\beta e_{1,12},
2i\chi_\beta e_{1,22},
\mu_{1\alpha} \delta^\beta_\alpha + 2i\chi_\beta e_{1,1\alpha}, 2i \chi_\beta e_{1,2\alpha}, 2i \chi_\beta e_{1,33}),\\
B_{1,2\beta}&=(2i \chi_\beta e_{2,11}, 2i \chi_\beta e_{2,12},
2i\chi_\beta e_{2,22},
2i\chi_\beta e_{2,1\alpha}, \mu_{2\alpha} \delta^\beta_\alpha + 2i \chi_\beta e_{2,2\alpha}, 2i \chi_\beta e_{2,33}),\\
B_{1,33} & = ({\cal L}_1^3 \phi_{11}, {\cal L}_1^3 \phi_{12}, {\cal
L}_1^3 \phi_{22},
{\cal L}_1^3 \phi_{1\alpha},{\cal L}_1^3 \phi_{2\alpha},  {\cal L}_1^3 \phi_{33}) |_{(0,0)}.\\
\end{split}
\end{equation}
By (\ref{phi (3,0)})-(\ref{phi 2}), $B_{1,33}$ can be calculated as
follows.

\begin{equation*}
\begin{split}
&{\cal L}_1^3 \phi_{11}|_{(0,0)} = 6\cdot 2 i \ov{e_{1,11}} + 6i
\chi_1 \cdot 2\cdot 2i
\bigg(\frac{\sqrt{\mu_1}}{\mu_1} \ov{e_{1,11}}e_{1,11} + \frac{\sqrt{\mu_1}}{\mu_2}\ov{e_{2,11}}e_{2,11} \\
&\ \ -\frac{1}{\sqrt{\mu_1}} f_1^{(I_1+2I_n)}\bigg)
 -12 \chi_1^2 4i\bigg(-\frac{1}{\mu_1}e_{1,11} f_1^{(I_1+2I_n)} - \frac{1}{\mu_2} e_{2,11} f_2^{(I_1+2I_n)}\bigg)\\
& = - \frac{1}{\sqrt{\mu_1} \chi_1} C_1 - 2i \bigg(\frac{e_{1,11}}{\mu_1} C_1 + \frac{e_{2,11}}{\mu_2} C_2 \bigg),\\
\end{split}
\end{equation*}

\begin{equation*}
\begin{split}
&{\cal L}_1^3 \phi_{12}|_{(0,0)} = 6
\frac{\sqrt{\mu_1}}{\sqrt{\mu_1+\mu_2}} \cdot 2 i \ov{e_{2,11}} + 6i
\chi_1 \cdot 2\cdot 2i
\bigg(\frac{\sqrt{\mu_1}}{\mu_1} \ov{e_{1,11}}e_{1,12} + \frac{\sqrt{\mu_1}}{\mu_2}\ov{e_{2,11}}e_{2,12} \\
&\ \ -\frac{1}{\sqrt{\mu_1+\mu_2}} f_2^{(I_1+2I_n)}\bigg)
-12 \chi_1^2 4i\bigg(-\frac{1}{\mu_1}e_{1,12} f_1^{(I_1+2I_n)} - \frac{1}{\mu_2} e_{2,12} f_2^{(I_1+2I_n)}\bigg)\\
& = - \frac{1}{\sqrt{\mu_1+\mu_2} \chi_1} C_2 - 2i \bigg(\frac{e_{1,12}}{\mu_1} C_1 + \frac{e_{2,12}}{\mu_2} C_2 \bigg),\\
\end{split}
\end{equation*}

\begin{equation*}
\begin{split}
&{\cal L}_1^3 \phi_{22}|_{(0,0)} = 6i \chi_1 \cdot 2\cdot 2i
\bigg(\frac{\sqrt{\mu_1}}{\mu_1} \ov{e_{1,11}}e_{1,22} + \frac{\sqrt{\mu_1}}{\mu_2}\ov{e_{2,11}}e_{2,22}\bigg)\\
&-12 \chi_1^2 4i\bigg(-\frac{1}{\mu_1}e_{1,22} f_1^{(I_1+2I_n)} -
\frac{1}{\mu_2} e_{2,22} f_2^{(I_1+2I_n)}\bigg)
=  - 2i \bigg(\frac{e_{1,22}}{\mu_1} C_1 + \frac{e_{2,22}}{\mu_2} C_2 \bigg).\\
\end{split}
\end{equation*}

For $j=1$ or $2$, we have

\begin{equation*}
\begin{split}
&{\cal L}_1^3 \phi_{j\alpha}|_{(0,0)} = 6i \chi_1 \cdot 2\cdot 2i
\bigg(\frac{\sqrt{\mu_1}}{\mu_1} \ov{e_{1,11}}e_{1,j\alpha} +
\frac{\sqrt{\mu_1}}{\mu_2}\ov{e_{2,11}}e_{2,j\alpha}
\bigg)\\
&\ \ \ -12 \chi_1^2 4i\bigg(-\frac{1}{\mu_1}e_{1,j\alpha}
f_1^{(I_1+2I_n)} - \frac{1}{\mu_2} e_{2,j\alpha}
f_2^{(I_1+2I_n)}\bigg)
 = - 2i \bigg(\frac{e_{1,j\alpha}}{\mu_1} C_1 + \frac{e_{2,j\alpha}}{\mu_2} C_2 \bigg).\\
\end{split}
\end{equation*}

We also have
\begin{equation*}
\begin{split}
&{\cal L}_1^3 \phi_{33}|_{(0,0)} = 6\cdot
\frac{2\mu_1}{\sqrt{\mu_2(\mu_1+\mu_2)}} \ov{e_{2,11}} + 6i \chi_1
\cdot 2\cdot 2i\bigg(\frac{\sqrt{\mu_1}}{\mu_1} \ov{e_{1,11}}\cdot
e_{1,33}
+ \frac{\sqrt{\mu_1}}{\mu_2} \ov{e_{2,11}} \cdot e_{2,33}\\
&\ \ + \frac{\sqrt{\mu_1} i}{\sqrt{\mu_2(\mu_1+\mu_2)}}
f_2^{(I_1+2I_n)}\bigg) - 12 \chi_1^2 \cdot 2\cdot
2i\bigg(-\frac{1}{\mu_1}e_{1,33}f_1^{(I_1+2I_n)} - \frac{1}{\mu_2}
e_{2,33}f_2^{(I_1+2I_n)}\bigg)\\& = \frac{\sqrt{\mu_1}
i}{\sqrt{\mu_2(\mu_1+\mu_2)} \chi_1} C_2
 - 2i
\bigg(\frac{e_{1,33}}{\mu_1} C_1+\frac{e_{2,33}}{\mu_2} C_2\bigg).
\end{split}
\end{equation*}

\medspace

We write $B_2=(Id-G_1)B_1$ and $A_2=(Id-G_1)A_1$ so that
\begin{equation}
\label{new form b} B_2 \ov{\phi}(\chi, 0)^t=A_2
\end{equation}
where
\[
G_1:=
\begin{pmatrix}
0 & 0 & 0 & \frac{2\chi_1}{\chi_3} & 0 & 0 & 0 & 0 \\
0 & 0 & 0 & \frac{\chi_2}{\chi_3} & 0 & \frac{\chi_1}{\chi_3} & 0 & 0 \\
0 & 0 & 0 & 0 & 0 & \frac{2\chi_2}{\chi_3} & 0 & 0 \\
0 & 0 & 0 & 0 & 0 & 0 & 0 & 0 \\
0 & 0 & 0 & \frac{\chi_\alpha}{\chi_3} & 0 & 0 & 0 & 0 \\
0 & 0 & 0 & 0 & 0 & 0 & 0 & 0 \\
0 & 0 & 0 & 0 & 0 & \frac{\chi_\alpha}{\chi_3} & 0 & 0 \\
0 & 0 & 0 & \frac{-C_1}{\mu_1 \chi_3} & 0 & \frac{-C_2}{\mu_2 \chi_3} & 0 & 0 \\
\end{pmatrix}.
\]
Then
\begin{equation*}
\begin{split}
& B_{2,11}=B_{1,11} - \frac{2\chi_1}{\chi_3} B_{1,13}=(2\mu_{11}, 0,
0, -\frac{2\chi_1}{\chi_3}
\mu_{13}, 0, ..., 0),\\
& B_{2,12}=B_{1,12} -\bigg(\frac{\chi_2}{\chi_3} B_{1,13}
+ \frac{\chi_1}{\chi_3} B_{1,23} \bigg)=(0, \mu_{12}, 0, -\frac{\chi_2}{\chi_3} \mu_{13}, 0, ..., 0, -\frac{\chi_1}{\chi_3} \mu_{23}, 0, ..., 0), \\
& B_{2,22} = B_{1,22} - \frac{2\chi_2}{\chi_3} B_{1,23}=(0,0,
2\mu_{22}, 0, ..., 0,
-\frac{2\chi_2}{\chi_3} \mu_{23}, 0, ..., 0),\\
& B_{2,13}=B_{1,13},\\
& B_{2,23}=B_{1,23}, \\
\end{split}
\end{equation*}
\begin{equation*}
\begin{split}
& B_{2,1\alpha}=B_{1,1\alpha} - \frac{\chi_\alpha}{\chi_3} B_{1,13}
=(0,0,0, -\frac{\chi_\alpha}{\chi_3} \mu_{13}, 0, .., 0, \mu_{1\alpha}, 0, ..., 0),\\
& B_{2,2\alpha}=B_{1,2\alpha} - \frac{\chi_\alpha}{\chi_3} B_{1,23}
=(0,0,0, 0, ..., 0, -\frac{\chi_\alpha}{\chi_3} \mu_{23}, 0, .., 0, \mu_{2\alpha}, 0, ..., 0),\\
& B_{2,33} = B_{1,33} + \frac{B_{1,13}}{\mu_1 \chi_3} C_1 + \frac{B_{1,23}}{\mu_2 \chi_3} C_2\\
& =\bigg(\frac{-C_1}{\sqrt{\mu_1}\chi_1} ,
\frac{-C_2}{\sqrt{\mu_1+\mu_2}\chi_1}, 0,
\frac{C_1}{\sqrt{\mu_1}\chi_3}, 0, ..., 0, \frac{C_2}{\sqrt{\mu_2}
\chi_3}, 0, ..., 0, \frac{\sqrt{\mu_1} i
C_2}{\sqrt{\mu_2(\mu_1+\mu_2)}\chi_1}\bigg),
\end{split}
\end{equation*}
and
\begin{equation}
\begin{split}
& A_{2,11} = 2\mu_1 \chi_1^2 - \frac{2 \chi_1}{\chi_3} \mu_1 \chi_1 \chi_3 = 0,\\
& A_{2,12} = (\mu_1+\mu_2) \chi_1 \chi_2 -
\bigg(\frac{\chi_2}{\chi_3} \mu_1 \chi_1\chi_3
+ \frac{\chi_1}{\chi_3} \mu_2 \chi_2 \chi_3 \bigg) = 0,\\
& A_{2,22} = 2 \mu_2 \chi_2^2 - \frac{2 \chi_2}{\chi_3} \mu_2 \chi_2 \chi_3 = 0,\\
& A_{2,13} = A_{1,13},\ \ A_{2,23} = A_{1, 23}, \\
& A_{2,1\alpha} = \mu_1 \chi_1 \chi_\alpha - \frac{\chi_\alpha}{\chi_3} \mu_1 \chi_1 \chi_3 = 0,\\
& A_{2,2\alpha} = \mu_2 \chi_2 \chi_\alpha - \frac{\chi_\alpha}{\chi_3} \mu_2 \chi_2 \chi_3 = 0,\\
& A_{2, 33} = - \chi_1 C_1 - \chi_2 C_2 + \frac{\mu_1 \chi_1
\chi_3}{\mu_1 \chi_3} C_1 + \frac{\mu_2 \chi_2\chi_3}{\mu_2 \chi_3}
C_2 = 0,
\end{split}
\end{equation}
where $4\le \alpha\le n-1$. From (\ref{new form b}) we have
\begin{equation}
\begin{split}\label{rel}
&\ov{\phi_{11}}(\chi, 0)=\frac{\chi_1}{\chi_3} \ov{\phi_{13}}(\chi,
0),\ \ \ov{\phi_{12}}(\chi, 0)=\frac{\mu_{13}\chi_2}{\mu_{12}\chi_3}
\ov{\phi_{13}}(\chi, 0)
+ \frac{\mu_{23}\chi_1}{\mu_{12}\chi_3} \ov{\phi_{23}}(\chi, 0), \\
& \ov{\phi_{22}}(\chi, 0)=\frac{\chi_2}{\chi_3} \ov{\phi_{23}}(\chi,
0),\ \ \ov{\phi_{1\alpha}}(\chi, 0)=\frac{\chi_\alpha}{\chi_3}
\ov{\phi_{13}}(\chi, 0),\ \ \ov{\phi_{2\alpha}}(\chi,
0)=\frac{\chi_\alpha}{\chi_3} \ov{\phi_{23}}(\chi, 0),
\end{split}
\end{equation}
where $4\le \alpha\le n-1$. By considering $B_{2,13} \ov{\phi}(\chi,
0)^t=A_{2,13}$, we obtain
\begin{equation}
\begin{split}
& 2i \chi_3 e_{1,11} \ov{\phi_{11}} + 2i \chi_3 e_{1,12}
\ov{\phi_{12}}
+ 2i \chi_3 e_{1,22} \ov{\phi_{22}}+\mu_{13}\ov{\phi_{13}}\\
&+ \sum^{n-1}_{\alpha=3} 2i \chi_3 e_{1,1\alpha} \ov{\phi_{1\alpha}}
+ \sum^{n-1}_{\alpha=3} 2i \chi_3 e_{1,2\alpha} \ov{\phi_{2\alpha}}
+ 2i e_{1,33} \chi_3 \ov{\phi_{33}} = \mu_1 \chi_1 \chi_3.
\end{split}
\end{equation}

Substituting (\ref{rel}) into this equation, we get

\begin{equation} \label{2 13 eq}
\begin{split}
& \bigg[\mu_{13} + 2i\bigg(e_{1,11}\chi_1 +
\frac{\sqrt{\mu_1}}{\mu_{12}} e_{1,12} \chi_2
+ \sum^{n-1}_{\alpha=3} e_{1,1\alpha} \chi_\alpha \bigg) \bigg] \ov{\phi_{13}}\\
& + \bigg(e_{1,12} \chi_2 + \frac{\sqrt{\mu_2}}{\mu_{12}} e_{1,12}
\chi_1 + \sum^{n-1}_{\alpha=3} e_{1,2\alpha} \chi_\alpha\bigg)
\ov{\phi_{23}} + 2i e_{1,33} \chi_3 \ov{\phi_{33}} = \mu_1 \chi_1
\chi_3.
\end{split}
\end{equation}

\medspace

We can rewrite (\ref{2 13 eq}) as
\begin{equation} \label{e1}
(\mu_{13} + 2i \sqrt{\mu_1} E_1) \ov{\phi_{13}} + 2i \sqrt{\mu_2}
E_2 \ov{\phi_{23}} + 2i e_{1,33} \chi_3 \ov{\phi_{33}} = \mu_1
\chi_1 \chi_3
\end{equation}
where
\begin{equation*}
E_1= \frac{1}{\mu_{11}} e_{1,11}\chi_1 + \frac{1}{\mu_{12}} e_{1,12}
\chi_2 + \sum^{n-1}_{\alpha=3} \frac{1}{\mu_{1\alpha}} e_{1,1\alpha}
\chi_\alpha,\ \ E_2= \frac{1}{\mu_{22}} e_{1,22}\chi_2 +
\frac{1}{\mu_{12}} e_{1,12} \chi_1 + \sum^{n-1}_{\alpha=3}
\frac{1}{\mu_{2\alpha}} e_{1,2\alpha} \chi_\alpha.
\end{equation*}

\medspace

Similarly by considering $B_{2,23} \ov{\phi}(\chi, 0)^t = A_{2,23}$,
we obtain
\begin{equation} \label{e2}
\begin{split}
2i \sqrt{\mu_1} E_3 \ov{\phi_{13}} + (\mu_{23} + 2i \sqrt{\mu_2}
E_4) \ov{\phi_{23}} + 2i e_{2,33} \chi_3 \ov{\phi_{33}} = \mu_2
\chi_2 \chi_3
\end{split}
\end{equation}
where
\begin{equation*}
E_3= \frac{1}{\mu_{11}} e_{2,11}\chi_1 + \frac{1}{\mu_{12}} e_{2,12}
\chi_2 + \sum^{n-1}_{\alpha=3} \frac{1}{\mu_{1\alpha}} e_{2,1\alpha}
\chi_\alpha,\ \ E_4= \frac{1}{\mu_{22}} e_{2,22}\chi_2 +
\frac{1}{\mu_{12}} e_{2,12} \chi_1 + \sum^{n-1}_{\alpha=3}
\frac{1}{\mu_{2\alpha}} e_{2,2\alpha} \chi_\alpha.
\end{equation*}

\medspace

Substituting (\ref{rel}) into $B_{2,33} \ov{\phi}(\chi, 0)^t =
A_{2,33}$, we yield
\begin{equation}
\begin{split}
& -\frac{C_1}{\sqrt{\mu_1}\chi_1} \frac{\chi_1}{\chi_3}
\ov{\phi_{13}}(\chi, 0) -
\frac{C_2}{\sqrt{\mu_1+\mu_2}\chi_1}\bigg(\frac{\sqrt{\mu_1}
\chi_2}{\mu_{12} \chi_3} \ov{\phi_{13}}(\chi, 0)
+ \frac{\sqrt{\mu_2}\chi_1}{\mu_{12} \chi_3} \ov{\phi_{23}}(\chi, 0) \bigg)\\
& + \frac{C_1}{\sqrt{\mu_1}\chi_3} \ov{\phi_{13}}(\chi, 0) +
\frac{C_2}{\sqrt{\mu_2} \chi_3} \ov{\phi_{23}}(\chi, 0) +
\frac{\sqrt{\mu_1} i C_2}{\sqrt{\mu_2(\mu_1+\mu_2)}\chi_1}
\ov{\phi_{33}}(\chi, 0)=0.
\end{split}
\end{equation}
We notice that the polynomial $C_2\not=0$ because $e_{2,11}\not=0$.
Here we used the fact that $\phi_{33}^{(3I_1)}=\frac{2
\mu_1}{\sqrt{\mu_2(\mu_1+\mu_2)}} \ov{e_{2,11}}$ $\not=0$ in Case A.
Then divided by $C_2$, we obtain from above
\[
\frac{\sqrt{\mu_1}\chi_2}{\mu_1+\mu_2} \ov{\phi_{13}}(\chi, 0) -
\frac{\mu_1 \chi_1}{\sqrt{\mu_2} (\mu_1+\mu_2)} \ov{\phi_{23}}(\chi,
0) - \frac{\sqrt{\mu_1} i \chi_3}{\sqrt{(\mu_1+\mu_2)\mu_2}}
\ov{\phi_{33}}(\chi, 0) = 0,
\]
and hence
\begin{equation} \label{33 eq}
\ov{\phi_{33}}(\chi, 0) =
\frac{i}{\sqrt{\mu_1+\mu_2}\chi_3}(-\sqrt{\mu_2} \chi_2
\ov{\phi_{13}} + \sqrt{\mu_1} \chi_1 \ov{\phi_{23}})(\chi, 0).
\end{equation}
Then (\ref{e1}) and (\ref{e2}) further take the form at $(\chi, 0)$:
\begin{equation}
\begin{split}
&\bigg[\mu_{13}+2i\sqrt{\mu_1} E_1 +
2\sqrt{\frac{\mu_2}{\mu_1+\mu_2}} e_{1,33}
\chi_2\bigg]\ov{\phi_{13}} +\bigg[2i \sqrt{\mu_2} E_2 -
2\sqrt{\frac{\mu_1}{\mu_1+\mu_2}} e_{1,33}
\chi_1\bigg]\ov{\phi_{23}} =
\mu_1 \chi_1 \chi_3,\\
&\bigg[2i \sqrt{\mu_1} E_3 + 2\sqrt{\frac{\mu_2}{\mu_1+\mu_2}}
e_{2,33} \chi_2\bigg]\ov{\phi_{13}} + \bigg[\mu_{23}+2i\sqrt{\mu_2}
E_4 - 2\sqrt{\frac{\mu_1}{\mu_1+\mu_2}} e_{2,33}
\chi_1\bigg]\ov{\phi_{23}} = \mu_2 \chi_2 \chi_3.
\end{split}
\end{equation}
It implies that $\ov{\phi_{13}}(\chi, 0)$ and $\ov{\phi_{23}}(\chi,
0)$ take the form $\chi_3 \frac{P_1^{(2)}(\chi)}{P_2^{(2)}(\chi)}$
where $P_1^{(2)}(\chi)$ and $P_2^{(2)}(\chi)$ are polynomials of
degree 2. Substituting these forms back to (\ref{rel}) and (\ref{33
eq}),  we conclude that (\ref{reduction}) is proved and hence
$deg(F(z, 0))\le 3$.
 This completes the proof of Proposition \ref{propdeg} for Case A.
 Case A$'$ can be similarly proved.
\end{proof}

\section{Proof of Proposition \ref{propdeg} for Case B}

This section is devoted to  the proof of Proposition \ref{propdeg}
for Case B.

\begin{proof}[Proof of Proposition \ref{propdeg} for Case B]
In this case, we suppose  $\phi^{(2I_1+I_2)}_{33}\neq 0$. Thus
\[
\frac{\mathcal{L}_1^2\mathcal{L}_2  g}{2i}|_{(0, 0)} =
\sum^{n-1}_{l=1} \mathcal{L}_1^2\mathcal{L}_2 f_l|_{(0, 0)}
\ov{f_l(\ov\chi, 0)} + \sum_{(s, t)\in {\cal S}_0}
\mathcal{L}_1^2\mathcal{L}_2 \phi_{st}|_{(0, 0)}
\ov{\phi_{st}(\ov\chi, 0)} + \mathcal{L}_1^2\mathcal{L}_2
\phi_{33}|_{(0, 0)} \ov{\phi_{33}(\ov\chi,
    0)}.
\]
Write
\begin{equation}\begin{split}
D_1=&-4i\chi_1\cdot \mu_{12}\ov{e_{1,12}} +4i\chi_1\cdot
2i\chi_2\cdot f_1^{(I_1+2I_n)}+(2i\chi_1)^2\cdot
2f_1^{(I_2+2I_n)},\\
D_2=&-4i\chi_2\cdot \mu_{11}\ov{e_{1,11}} +4i\chi_1\cdot
2i\chi_2\cdot f_1^{(I_1+2I_n)},\\
D_3=&-4i\chi_1\cdot \mu_{12}\ov{e_{2,12}} +4i\chi_1\cdot
2i\chi_2\cdot f_2^{(I_1+2I_n)}+(2i\chi_1)^2\cdot
2f_2^{(I_2+2I_n)},\\
D_4=&-4i\chi_2\cdot \mu_{11}\ov{e_{2,11}} +4i\chi_1\cdot
2i\chi_2\cdot f_2^{(I_1+2I_n)}.
\end{split}\end{equation}
A direct computation shows that $\mathcal{L}_1^2\mathcal{L}_2 =
\frac{\p ^3}{\p z_1^2 \p
    z_2}+2i\chi_2\frac{\p ^3}{\p z_1^2 \p w}+4i\chi_1\frac{\p ^3}{\p z_1
    \p z_2 \p w}+4i\chi_1\cdot 2i\chi_2\frac{\p ^3}{\p z_1 \p
    w^2}
+(2i\chi_1)^2\frac{\p ^3}{\p z_2 \p w^2}+(2i\chi_1)^2\cdot
2i\chi_2\frac{\p ^3}{ \p w^3}$. Thus
\begin{equation}
\begin{split}
& \mathcal{L}_1^2\mathcal{L}_2 f_1|_{(0,0)} =
2i\chi_2\cdot 2f_1^{(2I_1+I_n)}+4i\chi_1f_1^{(I_1+I_2+I_n)}\\
&\ +4i\chi_1\cdot 2i\chi_2\cdot 2f_1^{(I_1+2I_n)}
 + (2i\chi_1)^2\cdot 2f_1^{(I_2+2I_n)} = D_1+D_2.
\end{split}
\end{equation}
\begin{equation}\begin{split}
&\mathcal{L}_1^2\mathcal{L}_2 f_2|_{(0,0)} = 2i\chi_2\cdot
2f_2^{(2I_1+I_n)}+4i\chi_1f_2^{(I_1+I_2+I_n)}\\
&+4i\chi_1\cdot 2i\chi_2\cdot 2f_2^{(I_1+2I_n)}+(2i\chi_1)^2\cdot
2f_2^{(I_2+2I_n)} = D_3+D_4.
\end{split}\end{equation}
As in (\ref{new form}), we have
\begin{equation}
\label{new formb app} {\cal B}_1 \ov{\phi}(\chi, 0)^t={\cal A}_1.
\end{equation}
where
\begin{equation}
{\cal B}_1:= \begin{pmatrix}
{\cal L}_j{\cal L}_k \phi_{hl} & {\cal L}_j{\cal L}_k \phi_{h\alpha} & {\cal L}_j{\cal L}_k \phi_{33} \\
{\cal L}_j{\cal L}_\beta \phi_{hl} & {\cal L}_j{\cal L}_\beta
\phi_{h\alpha} & {\cal L}_j{\cal L}_\beta
\phi_{33}\\
{\cal L}^2_1 {\cal L}_2\phi_{hl} & {\cal L}^2_1 {\cal L}_2
\phi_{h\alpha} &
{\cal L}^2_1 {\cal L}_2 \phi_{33} \\
\end{pmatrix}_{1\leq j,k,h,l\leq 2, {3\leq \alpha,\beta\leq n-1}} \bigg|_{(0,0)}
=\begin{pmatrix}
B_{1, jk}\\
B_{1, j \beta}\\
{\cal B}_{1, 33}
\end{pmatrix}.
\end{equation}
and ${\cal A}_1=(A_{1,11}, A_{1,12}, A_{1,22}, A_{1,1\alpha},
A_{1,2\alpha}, {\cal A}_{1,33})^t$ where ${\cal A}_{1,33}=-{\cal
L}_1^2 {\cal L}_2 f_1|_{(0,0)} \chi_1 - {\cal L}_1^2 {\cal L}_2
f_2|_{(0,0)}$ $\chi_2 = -(D_1+D_2) \chi_1  - (D_3+D_4) \chi_2$, and
${\cal B}_{1,33}=({\cal L}_1^2{\cal L}_2  \phi_{hl}|_{(0,0)},\ {\cal
L}_1^2{\cal L}_2 \phi_{h\alpha}|_{(0,0)},\ {\cal L}_1^2{\cal L}_2
\phi_{33}|_{(0,0)})$. Recall
$\mathcal{L}_1^2\mathcal{L}_2\phi|_{(0,0)} =
2\phi^{(2I_1+I_2)}+2i\chi_2\cdot
2\phi^{(2I_1+I_n)}+4i\chi_1\phi^{(I_1+I_2+I_n)} +4i\chi_1\cdot
2i\chi_2\cdot 2\phi^{(I_1+2I_n)}+(2i\chi_1)^2\cdot
2\phi^{(I_2+2I_n)}$. We calculate ${\cal B}_{1,33}$ in details as
follows.

\begin{equation}\begin{split}
& \mathcal{L}_1^2\mathcal{L}_2\phi_{11}|_{(0,0)} =  2\cdot
\frac{2i}{\sqrt{\mu_1}}\mu_{12}\ov{e_{1,12}}+2i\chi_2\cdot
2\Big(2i(\frac{\mu_{11}}{\mu_1}\ov{e_{1,11}}e_{1,11}
+\frac{\mu_{11}}{\mu_2}\ov{e_{2,11}}e_{2,11}) \\
&\ \ \ -\frac{2i}{\sqrt{\mu_1}}f_1^{(I_1+2I_n)}\Big)
 + 4i\chi_1\cdot \Big(2i(\frac{\mu_{12}}{\mu_1}\ov{e_{1,12}}e_{1,11}
+\frac{\mu_{12}}{\mu_2}\ov{e_{2,12}}e_{2,11})-\frac{2i}{\sqrt{\mu_1}}f_1^{(I_2+2I_n)}\Big)\\
&\ \ \ + 4i\chi_1\cdot 2i\chi_2\cdot
2\Big(-\frac{2i}{{\mu_1}}e_{1,11}f_1^{(I_1+2I_n)}-\frac{2i}{{\mu_2}}e_{2,11}f_2^{(I_1+2I_n)}\Big)\\
&\ \ \ + (2i\chi_1)^2\cdot
2\Big(-\frac{2i}{{\mu_1}}e_{1,11}f_1^{(I_2+2I_n)}-\frac{2i}{{\mu_2}}e_{2,11}f_2^{(I_2+2I_n)}\Big)\\
& =
-\frac{2i}{\mu_1}e_{1,11}(D_1+D_2)-\frac{2i}{\mu_2}e_{2,11}(D_3+D_4)-\frac{D_1}{\sqrt{\mu_1}\chi_1},
\end{split}\end{equation}

\begin{equation}\begin{split}
& \mathcal{L}_1^2\mathcal{L}_2\phi_{12}|_{(0,0)} =  2\cdot
\frac{2i}{{\mu_{12}}}(\mu_{12}\ov{e_{2,12}}+\mu_{11}\ov{e_{1,11}}\big)\\
&\ \ \ + 2i\chi_2\cdot
2\Big(2i(\frac{\mu_{11}}{\mu_1}\ov{e_{1,11}}e_{1,12}
+\frac{\mu_{11}}{\mu_2}\ov{e_{2,11}}e_{2,12})-\frac{2i}{{\mu_{12}}}f_2^{(I_1+2I_n)}\Big)\\
&\ \ \ + 4i\chi_1\cdot
\Big(2i(\frac{\mu_{12}}{\mu_1}\ov{e_{1,12}}e_{1,12}
+\frac{\mu_{12}}{\mu_2}\ov{e_{2,12}}e_{2,12})-\frac{2i}{\mu_{12}}(f_1^{(I_1+2I_n)}+f_2^{(I_2+2I_n)})\Big)\\
&\ \ \ + 4i\chi_1\cdot 2i\chi_2\cdot
2\Big(-\frac{2i}{{\mu_1}}e_{1,12}f_1^{(I_1+2I_n)}-\frac{2i}{{\mu_2}}e_{2,12}f_2^{(I_1+2I_n)}\Big)\\
&\ \ \ + (2i\chi_1)^2\cdot
2\Big(-\frac{2i}{{\mu_1}}e_{1,12}f_1^{(I_2+2I_n)}-\frac{2i}{{\mu_2}}e_{2,12}f_2^{(I_2+2I_n)}\Big)\\
\end{split}\end{equation}

\begin{equation*}\begin{split}
& =
-\frac{2i}{{\mu_1}}e_{1,12}(D_1+D_2)-\frac{2i}{\mu_2}e_{2,12}(D_3+D_4)
-\frac{1}{\mu_{12}\chi_1}D_3-\frac{1}{\mu_{12}\chi_2}D_2,
\end{split}\end{equation*}

\begin{equation}\begin{split}
& \mathcal{L}_1^2\mathcal{L}_2\phi_{22}|_{(0,0) } = 2\cdot
\frac{2i}{{\mu_{22}}}\mu_{11}\ov{e_{2,11}}+2i\chi_2\cdot 2\cdot
2i\big(\frac{\mu_{11}}{\mu_1}\ov{e_{1,11}}e_{1,22}
+\frac{\mu_{11}}{\mu_2}\ov{e_{2,11}}e_{2,22}\big)\\
&\ \ \ + 4i\chi_1\cdot
\Big(2i(\frac{\mu_{12}}{\mu_1}\ov{e_{1,12}}e_{1,22}
+\frac{\mu_{12}}{\mu_2}\ov{e_{2,12}}e_{2,22})-\frac{2i}{\sqrt{\mu_{2}}}f_2^{(I_1+2I_n)}\Big)\\
&\ \ \ + 4i\chi_1\cdot 2i\chi_2\cdot
2\Big(-\frac{2i}{{\mu_1}}e_{1,22}f_1^{(I_1+2I_n)}-\frac{2i}{{\mu_2}}e_{2,22}f_2^{(I_1+2I_n)}\Big)\\
&\ \ \ + (2i\chi_1)^2\cdot
2\Big(-\frac{2i}{{\mu_1}}e_{1,22}f_1^{(I_2+2I_n)}-\frac{2i}{{\mu_2}}e_{2,22}f_2^{(I_2+2I_n)}\Big)\\
& =
-\frac{2i}{{\mu_1}}e_{1,22}(D_1+D_2)-\frac{2i}{{\mu_2}}e_{2,22}(D_3+D_4)
-\frac{1}{\sqrt{\mu_{2}}\chi_2}D_4,
\end{split}\end{equation}

\begin{equation}\begin{split}
& \mathcal{L}_1^2\mathcal{L}_2\phi_{j\alpha}|_{(0,0)} =
2i\chi_2\cdot 2\cdot
2i(\frac{\mu_{11}}{\mu_1}\ov{e_{1,11}}e_{1,j\alpha}
+\frac{\mu_{11}}{\mu_2}\ov{e_{2,11}}e_{2,j\alpha})\\
&\ \ \ + 4i\chi_1\cdot
2i(\frac{\mu_{12}}{\mu_1}\ov{e_{1,12}}e_{1,j\alpha}
+\frac{\mu_{12}}{\mu_2}\ov{e_{2,12}}e_{2,j\alpha})\\
&\ \ \ + 4i\chi_1\cdot 2i\chi_2\cdot
2\Big(-\frac{2i}{{\mu_1}}e_{1,j\alpha}f_1^{(I_1+2I_n)}
-\frac{2i}{{\mu_2}}e_{2,j\alpha}f_2^{(I_1+2I_n)}\Big)\\
&\ \ \ + (2i\chi_1)^2\cdot
2\Big(-\frac{2i}{{\mu_1}}e_{1,j\alpha}f_1^{(I_2+2I_n)}
-\frac{2i}{{\mu_2}}e_{2,j\alpha}f_2^{(I_2+2I_n)}\Big)\\
& =
-\frac{2i}{{\mu_1}}e_{1,j\alpha}(D_1+D_2)-\frac{2i}{{\mu_2}}e_{2,j\alpha}(D_3+D_4),
\end{split}\end{equation}
for $j=1, 2$, and
\begin{equation}\begin{split}
&\mathcal{L}_1^2\mathcal{L}_2\phi_{33}|_{(0,0)} = 2\cdot
\frac{2}{\mu_{12}}\big(\sqrt{\frac{\mu_1}{\mu_2}}\mu_{12}\ov{e_{2,12}}
-\sqrt{\frac{\mu_2}{\mu_1}}\mu_{11}\ov{e_{1,11}}\big)\\
&\ \ \ +2i\chi_2\cdot
2\Big(2i(\frac{\mu_{11}}{\mu_1}\ov{e_{1,11}}e_{1,33}
+\frac{\mu_{11}}{\mu_2}\ov{e_{2,11}}e_{2,33})-\frac{2}{{\mu_{12}}}
\sqrt{\frac{\mu_1}{\mu_2}}f_2^{(I_1+2I_n)}\Big)\\
&\ \ \ + 4i\chi_1\cdot
\Big(2i(\frac{\mu_{12}}{\mu_1}\ov{e_{1,12}}e_{1,33}
+\frac{\mu_{12}}{\mu_2}\ov{e_{2,12}}e_{2,33})-\frac{2}{\mu_{12}}(\sqrt{\frac{\mu_1}{\mu_2}}f_2^{(I_2+2I_n)}
-\sqrt{\frac{\mu_2}{\mu_1}}f_1^{(I_1+2I_n)})\Big)\\
&\ \ \ + 4i\chi_1\cdot 2i\chi_2\cdot
2\Big(-\frac{2i}{{\mu_1}}e_{1,33}f_1^{(I_1+2I_n)}-\frac{2i}{{\mu_2}}e_{2,33}f_2^{(I_1+2I_n)}\Big)\\
&\ \ \ + (2i\chi_1)^2\cdot
2\Big(-\frac{2i}{{\mu_1}}e_{1,33}f_1^{(I_2+2I_n)}-\frac{2i}{{\mu_2}}e_{2,33}f_2^{(I_2+2I_n)}\Big)\\
\end{split}\end{equation}
\begin{equation*}\begin{split}
=&\ \ \
-\frac{2i}{{\mu_1}}e_{1,33}(D_1+D_2)-\frac{2i}{{\mu_2}}e_{2,33}(D_3+D_4)
+\frac{i}{\mu_{12}\chi_1} \sqrt{\frac{\mu_1}{\mu_2}}
D_3-\frac{i}{\mu_{12}\chi_2} \sqrt{\frac{\mu_2}{\mu_1}}  D_2.
\end{split}\end{equation*}

\medspace

We write ${\cal B}_2=(Id-{\cal G}_1){\cal B}_1$ and ${\cal
A}_2=(Id-{\cal G}_1){\cal A}_1$ so that
\[
{\cal B}_2 \ov{\phi}(\chi, 0)^t = {\cal A}_2
\]
where
\[
{\cal G}_1:=
\begin{pmatrix}
0 & 0 & 0 & \frac{2\chi_1}{\chi_3} & 0 & 0 & 0 & 0 \\
0 & 0 & 0 & \frac{\chi_2}{\chi_3} & 0 & \frac{\chi_1}{\chi_3} & 0 & 0 \\
0 & 0 & 0 & 0 & 0 & \frac{2\chi_2}{\chi_3} & 0 & 0 \\
0 & 0 & 0 & 0 & 0 & 0 & 0 & 0 \\
0 & 0 & 0 & \frac{\chi_\alpha}{\chi_3} & 0 & 0 & 0 & 0 \\
0 & 0 & 0 & 0 & 0 & 0 & 0 & 0 \\
0 & 0 & 0 & 0 & 0 & \frac{\chi_\alpha}{\chi_3} & 0 & 0 \\
0 & 0 & 0 & -\frac{D_1+D_2}{\mu_1 \chi_3} & 0 & -\frac{D_3+D_4}{\mu_2 \chi_3} & 0 & 0 \\
\end{pmatrix}.
\]
By the construction of ${\cal B}_2$ and ${\cal A}_2$, we see that
(\ref{rel})(\ref{e1}) and (\ref{e2}) still hold. We further
calculate:
\begin{equation}\begin{split}
{\cal B}_{2,33} & ={\cal B}_{1,33}+\frac{B_{1,13}}{\mu_1\chi_3}(D_1+D_2)+\frac{B_{1,23}}{\mu_2\chi_3}(D_3+D_4)\\
& =
\Big(-\frac{D_1}{\sqrt{\mu_1}\chi_1},-\frac{D_3}{\mu_{12}\chi_1}-\frac{D_2}{\mu_{12}\chi_2},
-\frac{D_4}{\sqrt{\mu_2}\chi_2},\frac{D_1+D_2}{\sqrt{\mu_1}\chi_3},0,\cdots,0,\\
&\frac{D_3+D_4}{\sqrt{\mu_2}\chi_3},
0,\cdots,0,\frac{i}{\mu_{12}}\sqrt{\frac{\mu_1}{\mu_2}}\frac{D_3}{\chi_1}
-\frac{i}{\mu_{12}}\sqrt{\frac{\mu_2}{\mu_1}}\frac{D_2}{\chi_2}\Big).
\end{split}\end{equation}
and
\begin{equation*}\begin{split}
& {\cal A}_{2,33} =
-\chi_1(D_1+D_2)-\chi_2(D_3+D_4)+\frac{\mu_1\chi_1\chi_3}{\mu_1\chi_3}(D_1+D_2)
+\frac{\mu_2\chi_2\chi_3}{\mu_2\chi_3}(D_3+D_4) = 0.
\end{split}\end{equation*}

\medspace

We turn to ${\cal B}_{2,33} \ov\phi(\chi, 0)^t = {\cal A}_{2,33}$ to
have
\begin{equation}\begin{split}
&-\frac{D_1}{\sqrt{\mu_1}\chi_1}\ov{\phi}_{11}-\Big(\frac{D_3}{\mu_{12}\chi_1}+\frac{D_2}{\mu_{12}\chi_2}\Big)\ov{\phi}_{12}
-\frac{D_4}{\sqrt{\mu_2}\chi_2}\ov{\phi}_{22}\\
&+\frac{D_1+D_2}{\sqrt{\mu_1}\chi_3}\ov{\phi}_{13}+\frac{D_3+D_4}{\sqrt{\mu_2}\chi_3}
\ov{\phi}_{23}+\Big(\frac{i}{\mu_{12}}\sqrt{\frac{\mu_1}{\mu_2}}\frac{D_3}{\chi_1}
-\frac{i}{\mu_{12}}\sqrt{\frac{\mu_2}{\mu_1}}\frac{D_2}{\chi_2}\Big)\ov{\phi}_{33}=0.
\end{split}\end{equation}
Substituting (\ref{rel}) to this equation, we obtain

\begin{equation}\begin{split}
&-\frac{D_1}{\sqrt{\mu_1}\chi_3}\ov{\phi}_{13}-\Big(\frac{D_3}{\mu_{12}\chi_1}+\frac{D_2}{\mu_{12}\chi_2}\Big)\cdot
\Big(\frac{\mu_{13}\chi_2}{\mu_{12}\chi_3}\ov{\phi_{13}}+\frac{\mu_{23}\chi_1}{\mu_{12}\chi_3}\ov{\phi_{23}}\Big)
-\frac{D_4}{\sqrt{\mu_2}\chi_3}\ov{\phi}_{23}\\
&+\frac{D_1+D_2}{\sqrt{\mu_1}\chi_3}\ov{\phi}_{13}+\frac{D_3+D_4}{\sqrt{\mu_2}\chi_3}
\ov{\phi}_{23}+\Big(\frac{i}{\mu_{12}}\sqrt{\frac{\mu_1}{\mu_2}}\frac{D_3}{\chi_1}
-\frac{i}{\mu_{12}}\sqrt{\frac{\mu_2}{\mu_1}}\frac{D_2}{\chi_2}\Big)\ov{\phi}_{33}=0.
\end{split}\end{equation}
A quick simplification gives
\begin{equation}\begin{split}
&
\frac{\mu_2\chi_1D_2-\mu_1\chi_2D_3}{\sqrt{\mu_1}(\mu_1+\mu_2)\chi_1\chi_3}\ov{\phi}_{13}
-\frac{\mu_2\chi_1D_2-\mu_1\chi_2D_3}{\sqrt{\mu_2}(\mu_1+\mu_2)\chi_2\chi_3}\ov{\phi}_{23}
-\frac{i({\mu_2\chi_1D_2-\mu_1\chi_2D_3})}{\sqrt{\mu_1\mu_2(\mu_1+\mu_2)}\chi_1\chi_2}\ov{\phi}_{33}
=0.
\end{split}\end{equation}
Since $\phi^{(2I_1+I_2)}_{33}\neq 0$ in Case B, the polynomial
$\mu_2\chi_1D_2-\mu_1\chi_2D_3\neq 0$ because of (\ref{phi 1}). Thus
\begin{equation}\begin{split}
& \frac{1}{\sqrt{\mu_1}(\mu_1+\mu_2)\chi_1\chi_3}\ov{\phi}_{13}
-\frac{1}{\sqrt{\mu_2}(\mu_1+\mu_2)\chi_2\chi_3}\ov{\phi}_{23}
-\frac{i}{\sqrt{\mu_1\mu_2(\mu_1+\mu_2)}\chi_1\chi_2}\ov{\phi}_{33}
=0.
\end{split}\end{equation}
Hence we obtain the same formula (\ref{33 eq}). As in case A, we
show $deg(F)\le 3$. The proof for the case B$'$ is similar to the
case B.
\end{proof}

\section{Proof of Proposition \ref{propdeg} for Case C}

In this section, we will prove Proposition \ref{propdeg} for Case C.

\begin{proof}[Proof of Proposition \ref{propdeg} for Case C]
In Case C, we suppose  $\phi^{(I_1+I_2+I_j)}_{33}\neq 0$. Then
\[
\frac{{\cal L}_1 {\cal L}_2 {\cal L}_j  g}{2i}|_{(0, 0)} =
\sum^{n-1}_{l=1} {\cal L}_1 {\cal L}_2 {\cal L}_j f_l|_{(0, 0)}
\ov{f_l(\ov\chi, 0)} + \sum_{(s, t)\in {\cal S}_0} {\cal L}_1 {\cal
L}_2 {\cal L}_j \phi_{st}|_{(0, 0)} \ov{\phi_{st}(\ov\chi, 0)}
\]
\[
 + {\cal L}_1 {\cal L}_2 {\cal L}_j \phi_{33}|_{(0, 0)} \ov{\phi_{33}(\ov\chi, 0)}.
\]

Recall ${\cal L}_1 \mathcal{L}_2\mathcal{L}_j = \frac{\p^3}{\p z_1
z_2 \p z_j} + 2i\chi_1 \frac{\p^3}{\p z_2 z_j \p w} +
2i\chi_2\frac{\p^3}{\p z_1 \p z_j \p w} + 2i\chi_j\frac{\p^3}{\p z_1
\p z_2 \p w} - 4 \chi_2 \chi_j \frac{\p^3}{\p z_1 \p w^2} - 4 \chi_1
\chi_j \frac{\p^3}{\p z_2 \p w^2} - 4 \chi_1 \chi_2 \frac{\p^3}{\p
z_j \p w^2} - 8 i \chi_1\chi_2 \chi_j\frac{\p^3}{ \p w^3}$. Then
\begin{equation*}\begin{split}
& {\cal L}_1{\cal L}_2{\cal L}_j f_1|_{(0,0)}= 2i\chi_1
f_1^{(I_2+I_j+I_n)} + 2i\chi_2 f_1^{(I_1+I_j+I_n)}
+ 2i\chi_j f_1^{(I_1+I_2+I_n)} \\
& \ \ - 8 \chi_2 \chi_j f_1^{(I_1+2I_n)} - 8 \chi_1 \chi_j
f_1^{(I_2+2I_n)} = \mathtt{D}_1 + \mathtt{D}_2 + \mathtt{D}_3,
\end{split}
\end{equation*}

\begin{equation*}\begin{split}
& {\cal L}_1{\cal L}_2{\cal L}_j f_2|_{(0,0)}= 2i\chi_1
f_2^{(I_2+I_j+I_n)} + 2i\chi_2 f_2^{(I_1+I_j+I_n)}
+ 2i\chi_j f_2^{(I_1+I_2+I_n)} \\
& \ \ - 8 \chi_2 \chi_j f_2^{(I_1+2I_n)} - 8 \chi_1 \chi_j
f_2^{(I_2+2I_n)} =\mathtt{D}_4 + \mathtt{D}_5 + \mathtt{D}_6
\end{split}
\end{equation*}
where
\[
\begin{cases}
\mathtt{D}_1 = - 2i \chi_1 \sqrt{\mu_2} \ov{e_{1,2j}} - 4 \chi_1 \chi_j f_1^{(I_2+2I_n)},\\
\mathtt{D}_2 = -2i \chi_2 \sqrt{\mu_1} \ov{e_{1,1j}} - 4 \chi_2 \chi_j f_1^{(I_1+2I_n)},\\
\mathtt{D}_3 = - 2i \chi_j \sqrt{\mu_1+\mu_2} \ov{e_{1,12}} - 4
\chi_1 \chi_j f_1^{(I_2+2I_n)}
- 4\chi_2\chi_j f_1^{(I_1+2I_n)}, \\
\end{cases}
\]
\[
\begin{cases}
\mathtt{D}_4 = - 2i \chi_1 \sqrt{\mu_2} \ov{e_{2,2j}} - 4 \chi_1 \chi_j f_2^{(I_2+2I_n)},\\
\mathtt{D}_5 = -2i \chi_2 \sqrt{\mu_1} \ov{e_{2,1j}} - 4 \chi_2 \chi_j f_2^{(I_1+2I_n)} ,\\
\mathtt{D}_6 = - 2i \chi_j \sqrt{\mu_1+\mu_2} \ov{e_{2,12}} - 4
\chi_1 \chi_j f_2^{(I_2+2I_n)}
- 4\chi_2\chi_j f_2^{(I_1+2I_n)}.\\
\end{cases}
\]

As in (\ref{new form}), we have
\begin{equation}
\label{new formbbb} \mathtt{B}_1 \ov{\phi}(\chi, 0)^t=\mathtt{A}_1.
\end{equation}
where
\begin{equation}
\mathtt{B}_1:= \begin{pmatrix}
{\cal L}_j{\cal L}_k \phi_{hl} & {\cal L}_j{\cal L}_k \phi_{h\alpha} & {\cal L}_j{\cal L}_k \phi_{33} \\
{\cal L}_j{\cal L}_\beta \phi_{hl} & {\cal L}_j{\cal L}_\beta
\phi_{h\alpha} & {\cal L}_j{\cal L}_\beta
\phi_{33}\\
{\cal L}_1 {\cal L}_2 {\cal L}_j\phi_{hl} & {\cal L}^2_1 {\cal L}_2
\phi_{h\alpha} &
{\cal L}_1 {\cal L}_2 {\cal L}_j\phi_{33} \\
\end{pmatrix}_{1\leq j,k,h,l\leq 2, {3\leq \alpha,\beta\leq n-1}} \bigg|_{(0,0)}
=\begin{pmatrix}
B_{1, jk}\\
B_{1, j \beta}\\
\mathtt{B}_{1, 33}
\end{pmatrix}.
\end{equation}
and $\mathtt{A}_1=(A_{1,11}, A_{1,12}, A_{1,22}, A_{1,1\alpha},
A_{1,2\alpha}, \mathtt{A}_{1,33})^t$ where

\noindent $\mathtt{A}_{1,33}=-{\cal L}_1 {\cal L}_2 {\cal L}_j
f_1|_{(0,0)} \chi_1 - {\cal L}_1 {\cal L}_2 {\cal L}_j f_2|_{(0,0)}$
$\chi_2 = -(\mathtt{D}_1+\mathtt{D}_2+\mathtt{D}_3) \chi_1  -
(\mathtt{D}_4+\mathtt{D}_5+\mathtt{D}_6) \chi_2$, and

\noindent $\mathtt{B}_{1,33}=({\cal L}_1 {\cal L}_2  {\cal L}_j
\phi_{hl}|_{(0,0)},\ {\cal L}_1 {\cal L}_2 {\cal L}_j
\phi_{h\alpha}|_{(0,0)},\  {\cal L}_1 {\cal L}_2 {\cal L}_j
\phi_{33}|_{(0,0)})$. Recall ${\cal L}_1 \mathcal{L}_2 {\cal L}_j
\phi|_{(0,0)} = \phi^{(I_1+I_2+I_j)}$ $+ 2i\chi_1
\phi^{(I_2+I_j+I_n)} + 2i\chi_2 \phi^{(I_1+I_j+I_n)} + 2i\chi_j
\phi^{(I_1+I_2+I_n)} - 8 \chi_2 \chi_j \phi^{(I_1+2I_n)} - 8 \chi_1
\chi_j \phi^{(I_2+2I_n)}$. Then
\begin{equation}\begin{split}
& {\cal L}_1 \mathcal{L}_2 {\cal L}_j \phi_{11}|_{(0,0)} =
\frac{2i\sqrt{\mu_2}}{\sqrt{\mu_1}} \ov{e_{1,2j}} + 2i\chi_1
\bigg(\frac{2i}{\mu_1} \mu_{2j} \ov{e_{1,2j}} e_{1,11}
+ \frac{2i}{\mu_2} \mu_{2j}\ov{e_{2,2j}}e_{2,11}\bigg) \\
& + 2i\chi_2 \bigg(\frac{2i}{\mu_1} \mu_{2j} \ov{e_{1,1j}} e_{1,11}
+ \frac{2i}{\mu_2} \mu_{1j}\ov{e_{2,1j}}e_{2,11}\bigg) \\
& + 2i\chi_j \bigg(\frac{2i}{\mu_1} \mu_{12} \ov{e_{1,12}} e_{1,11}
+ \frac{2i}{\mu_2} \mu_{12}\ov{e_{2,12}}e_{2,11}
- \frac{2i}{\sqrt{\mu_1}} f_1^{(I_2+2I_n)}\bigg) \\
& - 8 \chi_2 \chi_j \bigg(-\frac{2i}{\mu_1} f_1^{(I_1+2I_n)}
e_{1,11}
- \frac{2i}{\mu_2} f_2^{(I_1+2I_n)} e_{2,11}\bigg)\\
& - 8 \chi_1 \chi_j \bigg(-\frac{2i}{\mu_1} f_1^{(I_2+2I_n)}
e_{1,11}
- \frac{2i}{\mu_2} f_2^{(I_2+2I_n)} e_{2,11}\bigg)\\
& = - \frac{2i}{\mu_1} e_{1,11}
(\mathtt{D}_1+\mathtt{D}_2+\mathtt{D}_3) - \frac{2i}{\mu_2} e_{2,11}
(\mathtt{D}_4+\mathtt{D}_5+\mathtt{D}_6) -
\frac{1}{\sqrt{\mu_1}\chi_1} \mathtt{D}_1,
\end{split}\end{equation}

\begin{equation*}\begin{split}
& {\cal L}_1 \mathcal{L}_2 {\cal L}_j \phi_{12}|_{(0,0)} = \frac{2 i
}{\sqrt{\mu_1+\mu_2}}
\Big(\sqrt{\mu_2}\ov{e_{2,2j}}+\sqrt{\mu_1}\ov{e_{1,1j}}\Big)
 + 2i\chi_1  \bigg(\frac{2i}{\mu_1} \mu_{2j} \ov{e_{1,2j}} e_{1,12}\\
\end{split}\end{equation*}
\begin{equation}\begin{split}
&\ \ \ + \frac{2i}{\mu_2} \mu_{2j} \ov{e_{2,2j}} e_{2,12}\bigg) +
2i\chi_2  \bigg(\frac{2i}{\mu_1} \mu_{1j} \ov{e_{1,1j}} e_{1,12}
+ \frac{2i}{\mu_2} \mu_{1j} \ov{e_{2,1j}} e_{2,12}\bigg)\\
& + 2i\chi_j \bigg(\frac{2i}{\mu_1} \mu_{12} \ov{e_{1,12}} e_{1,12}
+ \frac{2i}{\mu_2} \mu_{12} \ov{e_{2,12}} e_{2,12}
- \frac{2i}{\sqrt{\mu_1+\mu_2}}f_2^{(I_2+2I_n)} - \frac{2i}{\sqrt{\mu_1+\mu_2}}f_1^{(I_1+2I_n)}\bigg)\\
& - 8 \chi_2 \chi_j \bigg(-\frac{2i}{\mu_1} f_1^{(I_1+2I_n)}
e_{1,12}
-\frac{2i}{\mu_2} f_2^{(I_1+2I_n)} e_{2,12}\bigg)\\
& - 8 \chi_1 \chi_j \bigg(-\frac{2i}{\mu_1} f_1^{(I_2+2I_n)}
e_{1,12}
-\frac{2i}{\mu_2} f_2^{(I_2+2I_n)} e_{2,12}\bigg)\\
& = - \frac{2i}{\mu_1} e_{1,12}
(\mathtt{D}_1+\mathtt{D}_2+\mathtt{D}_3)
 - \frac{2i}{\mu_2} e_{2,12} (\mathtt{D}_4+\mathtt{D}_5+\mathtt{D}_6)
 - \frac{1}{\sqrt{\mu_1+\mu_2}\chi_2} \mathtt{D}_2
 - \frac{1}{\sqrt{\mu_1+\mu_2} \chi_1} \mathtt{D}_4,
\end{split}\end{equation}

\begin{equation}\begin{split}
& {\cal L}_1 \mathcal{L}_2 {\cal L}_j \phi_{22}|_{(0,0)} =
\frac{2i\sqrt{\mu_1}}{\sqrt{\mu_2}} \ov{e_{2,1j}} + 2i\chi_1
\bigg(\frac{2i}{\mu_1} \mu_{2j} \ov{e_{1,2j}} e_{1,22}
+ \frac{2i}{\mu_2} \mu_{2j}\ov{e_{2,2j}}e_{2,22}\bigg) \\
& + 2i\chi_2 \bigg(\frac{2i}{\mu_1} \mu_{2j} \ov{e_{1,1j}} e_{1,22}
+ \frac{2i}{\mu_2} \mu_{1j}\ov{e_{2,1j}}e_{2,22}\bigg) \\
& + 2i\chi_j \bigg(\frac{2i}{\mu_1} \mu_{12} \ov{e_{1,12}} e_{1,22}
+ \frac{2i}{\mu_2} \mu_{12}\ov{e_{2,12}}e_{2,22}
- \frac{2i}{\sqrt{\mu_2}} f_2^{(I_1+2I_n)}\bigg) \\
&
- 8 \chi_2 \chi_j \bigg(-\frac{2i}{\mu_1} f_1^{(I_1+2I_n)} e_{1,22} - \frac{2i}{\mu_2} f_2^{(I_1+2I_n)} e_{2,22}\bigg)\\
&
- 8 \chi_1 \chi_j \bigg(-\frac{2i}{\mu_1} f_1^{(I_2+2I_n)} e_{1,22} - \frac{2i}{\mu_2} f_2^{(I_2+2I_n)} e_{2,22}\bigg)\\
& = - \frac{2i}{\mu_1} e_{1,12}
(\mathtt{D}_1+\mathtt{D}_2+\mathtt{D}_3) - \frac{2i}{\mu_2} e_{2,12}
(\mathtt{D}_4+\mathtt{D}_5+\mathtt{D}_6) -
\frac{1}{\sqrt{\mu_2}\chi_2} \mathtt{D}_5,
\end{split}\end{equation}

\begin{equation}\begin{split}
& {\cal L}_1 \mathcal{L}_2 {\cal L}_j \phi_{1k}|_{(0,0)} =
\frac{2i\sqrt{\mu_1+\mu_2}}{\sqrt{\mu_1}} \delta_{jk}\ov{e_{1,12}} \\
& + 2i\chi_1 \bigg(\frac{2i}{\mu_1} \mu_{2j} \ov{e_{1,2j}} e_{1,1k}
+ \frac{2i}{\mu_2} \mu_{2j} \ov{e_{2,2j}}e_{2,1k}
-\frac{2i}{\sqrt{\mu_1}}\delta_{jk} f_1^{(I_2+2I_n)}\bigg)\\
& + 2i\chi_2 \bigg(\frac{2i}{\mu_1} \mu_{1j} \ov{e_{1,1j}} e_{1,1k}
+ \frac{2i}{\mu_2} \mu_{1j} \ov{e_{2,1j}}e_{2,1k}
-\frac{2i}{\sqrt{\mu_1}}\delta_{jk} f_1^{(I_1+2I_n)}\bigg)\\
& + 2i\chi_j \bigg(\frac{2i}{\mu_1} \mu_{12} \ov{e_{1,12}} e_{1,1k}
+ \frac{2i}{\mu_2} \mu_{12} \ov{e_{2,12}}e_{2,1k}\bigg)\\
& - 8 \chi_2 \chi_j\bigg(-\frac{2i}{\mu_1} f_1^{(I_1+2I_n)} e_{1,1k}
- \frac{2i}{\mu_2} f_2^{(I_1+2I_n)} e_{2,1k}\bigg)\\
\end{split}\end{equation}

\begin{equation*}\begin{split}
&- 8 \chi_1 \chi_j \bigg(-\frac{2i}{\mu_1} f_1^{(I_2+2I_n)} e_{1,1k}
- \frac{2i}{\mu_2} f_2^{(I_2+2I_n)} e_{2,1k}\bigg)\\
& = - \frac{2i}{\mu_1} e_{1,1k}
(\mathtt{D}_1+\mathtt{D}_2+\mathtt{D}_3) - \frac{2i}{\mu_2} e_{2,1k}
(\mathtt{D}_4+\mathtt{D}_5+\mathtt{D}_6) -
\frac{1}{\sqrt{\mu_1}\chi_j} \delta_{jk}\mathtt{D}_3,
\end{split}\end{equation*}

\begin{equation}\begin{split}
& {\cal L}_1 \mathcal{L}_2 {\cal L}_j \phi_{2k}|_{(0,0)} =
\frac{2i\sqrt{\mu_1+\mu_2}}{\sqrt{\mu_2}}\delta_{jk} \ov{e_{2,12}} \\
& + 2i\chi_1 \bigg(\frac{2i}{\mu_1} \mu_{2j} \ov{e_{1,2j}} e_{1,2k}
+ \frac{2i}{\mu_2} \mu_{2j} \ov{e_{2,2j}}e_{2,2k}
-\frac{2i}{\sqrt{\mu_2}}\delta_{jk} f_2^{(I_2+2I_n)}\bigg)\\
& + 2i\chi_2 \bigg(\frac{2i}{\mu_1} \mu_{1j} \ov{e_{1,1j}} e_{1,2k}
+ \frac{2i}{\mu_2} \mu_{1j} \ov{e_{2,1j}}e_{2,2k}
-\frac{2i}{\sqrt{\mu_2}} \delta_{jk}f_2^{(I_1+2I_n)}\bigg)\\
& + 2i\chi_j \bigg(\frac{2i}{\mu_1} \mu_{12} \ov{e_{1,12}} e_{1,2k}
+ \frac{2i}{\mu_2} \mu_{12} \ov{e_{2,12}}e_{2,2k}\bigg)\\
& - 8 \chi_2 \chi_j\bigg(-\frac{2i}{\mu_1} f_1^{(I_1+2I_n)} e_{1,2k}
- \frac{2i}{\mu_2} f_2^{(I_1+2I_n)} e_{2,2k}\bigg)\\
&- 8 \chi_1 \chi_j \bigg(-\frac{2i}{\mu_1} f_1^{(I_2+2I_n)} e_{1,2k}
- \frac{2i}{\mu_2} f_2^{(I_2+2I_n)} e_{2,2k}\bigg)\\
& = - \frac{2i}{\mu_1} e_{1,2k}
(\mathtt{D}_1+\mathtt{D}_2+\mathtt{D}_3) - \frac{2i}{\mu_2} e_{2,2k}
(\mathtt{D}_4+\mathtt{D}_5+\mathtt{D}_6) -
\frac{1}{\sqrt{\mu_2}\chi_j} \delta_{jk}\mathtt{D}_6,
\end{split}\end{equation}

\begin{equation}\begin{split}
& {\cal L}_1 \mathcal{L}_2 {\cal L}_j \phi_{33}|_{(0,0)}
= \frac{2}{\sqrt{\mu_1+\mu_2}} \Big(\sqrt{\mu_1}\ov{e_{2,2j}}-\sqrt{\mu_2}\ov{e_{1,1j}}\Big) \\
& + 2i\chi_1 \bigg(\frac{2i}{\mu_1} \mu_{2j} \ov{e_{1,2j}} e_{1,33}
+\frac{2i}{\mu_2} \mu_{2j} \ov{e_{2,2j}} e_{2,33} \bigg)\\
& + 2i\chi_2 \bigg(\frac{2i}{\mu_1} \mu_{1j} \ov{e_{1,1j}} e_{1,33}
+\frac{2i}{\mu_2} \mu_{1j} \ov{e_{2,1j}} e_{2,33} \bigg)\\
& + 2i\chi_j \bigg(\frac{2i}{\mu_1} \mu_{12} \ov{e_{1,12}} e_{1,33}
+\frac{2i}{\mu_2} \mu_{12} \ov{e_{2,12}} e_{2,33} \\
&\ \ \ \ \ - \frac{2}{\sqrt{\mu_1 \mu_2(\mu_1+\mu_2)}}
\Big(\mu_1f_2^{(I_2+2I_n)}
-\mu_2 f_1^{(I_1+2I_n)}\Big)\bigg)\\
& - 8 \chi_2 \chi_j \bigg(-\frac{2i}{\mu_1} f_1^{(I_1+2I_n)}
e_{1,33}
- \frac{2i}{\mu_2} f_2^{(I_1+2I_n)} e_{2,33}\bigg)\\
\end{split}\end{equation}
\begin{equation*}\begin{split}
& - 8 \chi_1 \chi_j \bigg(-\frac{2i}{\mu_1} f_1^{(I_2+2I_n)}
e_{1,33}
- \frac{2i}{\mu_2} f_2^{(I_2+2I_n)} e_{2,33}\bigg)\\
& = - \frac{2i}{\mu_1} e_{1,33}
(\mathtt{D}_1+\mathtt{D}_2+\mathtt{D}_3)
- \frac{2i}{\mu_2} e_{2,33} (\mathtt{D}_4+\mathtt{D}_5+\mathtt{D}_6)\\
& + \frac{i \sqrt{\mu_1}}{\sqrt{\mu_2(\mu_1+\mu_2)}\chi_1}
\mathtt{D}_4 - \frac{i
\sqrt{\mu_2}}{\sqrt{\mu_1(\mu_1+\mu_2)}\chi_2} \mathtt{D}_2.
\end{split}
\end{equation*}

We write $\mathtt{B}_2=(Id - \mathtt{G}_1)\mathtt{B}_1$ and
$\mathtt{A}_2=(Id-\mathtt{G}_1) \mathtt{A}_1$ so that
\[
\mathtt{B}_2 \ov\phi(\chi, 0)^t = \mathtt{A}_2
\]
where
\[
\mathtt{G}_1:=
\begin{pmatrix}
0 & 0 & 0 & \frac{2\chi_1}{\chi_3} & 0 & 0 & 0 & 0 \\
0 & 0 & 0 & \frac{\chi_2}{\chi_3} & 0 & \frac{\chi_1}{\chi_3} & 0 & 0 \\
0 & 0 & 0 & 0 & 0 & \frac{2\chi_2}{\chi_3} & 0 & 0 \\
0 & 0 & 0 & 0 & 0 & 0 & 0 & 0 \\
0 & 0 & 0 & \frac{\chi_\alpha}{\chi_3} & 0 & 0 & 0 & 0 \\
0 & 0 & 0 & 0 & 0 & 0 & 0 & 0 \\
0 & 0 & 0 & 0 & 0 & \frac{\chi_\alpha}{\chi_3} & 0 & 0 \\
0 & 0 & 0 & -\frac{\mathtt{D}_1+\mathtt{D}_2+\mathtt{D}_3}{\mu_1
\chi_3} & 0
& -\frac{\mathtt{D}_4+\mathtt{D}_5 + \mathtt{D}_6}{\mu_2 \chi_3} & 0 & 0 \\
\end{pmatrix}.
\]
By the construction of $\mathtt{B}_2$ and $\mathtt{A}_2$, we see
that (\ref{rel})(\ref{e1}) and (\ref{e2}) still hold. We further
calculate:
\begin{equation}\begin{split}
\mathtt{B}_{2,33} & =
\mathtt{B}_{1,33}+\frac{B_{1,13}}{\mu_1\chi_3}(\mathtt{D}_1+\mathtt{D}_2+\mathtt{D}_3)
+\frac{B_{1,23}}{\mu_2\chi_3}(\mathtt{D}_4+\mathtt{D}_5 +\mathtt{D}_6)\\
& = \Big(-\frac{\mathtt{D}_1}{\sqrt{\mu_1}\chi_1},\
-\frac{\mathtt{D}_4}{\mu_{12}\chi_1} -
\frac{\mathtt{D}_2}{\mu_{12}\chi_2},\
-\frac{\mathtt{D}_5}{\sqrt{\mu_2}\chi_2}, \
\frac{\mathtt{D}_1+\mathtt{D}_2}{\sqrt{\mu_1}\chi_3}, 0,\cdots,0,\\
&\frac{\mathtt{D}_4+\mathtt{D}_5}{\sqrt{\mu_2}\chi_3},0, \cdots,0,
\frac{i}{\mu_{12}}\sqrt{\frac{\mu_1}{\mu_2}}\frac{\mathtt{D}_4}{\chi_1}
-\frac{i}{\mu_{12}}\sqrt{\frac{\mu_2}{\mu_1}}\frac{\mathtt{D}_2}{\chi_2}\Big).
\end{split}\end{equation}
and
\begin{equation}\begin{split}
& \mathtt{A}_{2,33}
= -\chi_1(\mathtt{D}_1+\mathtt{D}_2+\mathtt{D}_3)-\chi_2(\mathtt{D}_4+\mathtt{D}_5+\mathtt{D}_6) \\
&\ \ +
\frac{\mu_1\chi_1\chi_3}{\mu_1\chi_3}(\mathtt{D}_1+\mathtt{D}_2+\mathtt{D}_3)
+\frac{\mu_2\chi_2\chi_3}{\mu_2\chi_3}(\mathtt{D}_4+\mathtt{D}_5+\mathtt{D}_6)
= 0.
\end{split}\end{equation}

\medspace

We turn to $\mathtt{B}_{2,33} \ov\phi(\chi, 0)^t =
\mathtt{A}_{2,33}$ to have
\begin{equation}\begin{split}
&-\frac{\mathtt{D}_1}{\sqrt{\mu_1}\chi_1}\ov{\phi}_{11}
-\Big(\frac{\mathtt{D}_4}{\mu_{12}\chi_1}+\frac{\mathtt{D}_2}{\mu_{12}\chi_2}\Big)\ov{\phi}_{12}
-\frac{\mathtt{D}_5}{\sqrt{\mu_2}\chi_2}\ov{\phi}_{22}\\
&+\frac{\mathtt{D}_1+\mathtt{D}_2}{\sqrt{\mu_1}\chi_3}\ov{\phi}_{13}
+\frac{\mathtt{D}_4+\mathtt{D}_5}{\sqrt{\mu_2}\chi_3}\ov{\phi}_{23}
+\Big(\frac{i}{\mu_{12}}\sqrt{\frac{\mu_1}{\mu_2}}\frac{\mathtt{D}_4}{\chi_1}
-\frac{i}{\mu_{12}}\sqrt{\frac{\mu_2}{\mu_1}}\frac{\mathtt{D}_2}{\chi_2}\Big)\ov{\phi}_{33}=0.
\end{split}\end{equation}
Substituting (\ref{rel}) to this equation, we obtain
\begin{equation}\begin{split}
&-\frac{\mathtt{D}_1}{\sqrt{\mu_1}\chi_3}\ov{\phi}_{13}
-\Big(\frac{\mathtt{D}_4}{\mu_{12}\chi_1}+\frac{\mathtt{D}_2}{\mu_{12}\chi_2}\Big)\cdot
\Big(\frac{\mu_{13}\chi_2}{\mu_{12}\chi_3}\ov{\phi_{13}}
+\frac{\mu_{23}\chi_1}{\mu_{12}\chi_3}\ov{\phi_{23}}\Big)
-\frac{\mathtt{D}_5}{\sqrt{\mu_2}\chi_3}\ov{\phi}_{23}\\
&+\frac{\mathtt{D}_1+\mathtt{D}_2}{\sqrt{\mu_1}\chi_3}\ov{\phi}_{13}
+\frac{\mathtt{D}_4+\mathtt{D}_5}{\sqrt{\mu_2}\chi_3}\ov{\phi}_{23}
+\Big(\frac{i}{\mu_{12}}\sqrt{\frac{\mu_1}{\mu_2}}\frac{\mathtt{D}_4}{\chi_1}
-\frac{i}{\mu_{12}}\sqrt{\frac{\mu_2}{\mu_1}}\frac{\mathtt{D}_2}{\chi_2}\Big)\ov{\phi}_{33}=0.
\end{split}\end{equation}
A quick simplification gives
\begin{equation}\begin{split}
&-\frac{\mu_1\chi_2 \mathtt{D}_4-\mu_2\chi_1
\mathtt{D}_2}{\sqrt{\mu_1}(\mu_1+\mu_2)\chi_1\chi_3} \ov{\phi}_{13}
+\frac{\mu_1\chi_2 \mathtt{D}_4-\mu_2\chi_1
\mathtt{D}_2}{\sqrt{\mu_2}(\mu_1+\mu_2)\chi_2\chi_3} \ov{\phi}_{23}
+\frac{i({\mu_1\chi_2 \mathtt{D}_4-\mu_2\chi_1
\mathtt{D}_2})}{\sqrt{\mu_1\mu_2(\mu_1+\mu_2)}
    \chi_1\chi_2}\ov{\phi}_{33}
=0.
\end{split}\end{equation}
Since $\phi_{33}^{(I_1+I_2+I_j)}\not=0$ in Case C, the polynomial
$\mu_2\chi_1 \mathtt{D}_2-\mu_1\chi_2 \mathtt{D}_4\neq 0$ because of
(\ref{phi 1}). Thus
\begin{equation}\begin{split}
&-\frac{1}{\sqrt{\mu_1}(\mu_1+\mu_2)\chi_1\chi_3}\ov{\phi}_{13}
+\frac{1}{\sqrt{\mu_2}(\mu_1+\mu_2)\chi_2\chi_3}\ov{\phi}_{23}
+\frac{i}{\sqrt{\mu_1\mu_2(\mu_1+\mu_2)}\chi_1\chi_2}\ov{\phi}_{33}
=0.
\end{split}\end{equation}
Hence we obtain the same formula (\ref{33 eq}). As in Case A, we can
further get $deg(F)\le 3$.
\end{proof}

\medspace

\section{Proof of Proposition \ref{propdeg} for Case D}

In this section, we will prove Proposition \ref{propdeg} for Case D.

\begin{proof}[Proof of Proposition \ref{propdeg} for Case D] In this case, we suppose  $\phi^{(2I_1+I_j)}_{33}\neq 0$. Then
\[
\frac{\mathcal{L}_1^2\mathcal{L}_j  g}{2i}|_{(0, 0)} =
\sum^{n-1}_{l=1} \mathcal{L}_1^2\mathcal{L}_j f_l|_{(0, 0)}
\ov{f_l(\ov\chi, 0)} + \sum_{(s, t)\in {\cal S}_0}
\mathcal{L}_1^2\mathcal{L}_j \phi_{st}|_{(0, 0)}
\ov{\phi_{st}(\ov\chi, 0)} + \mathcal{L}_1^2\mathcal{L}_j
\phi_{33}|_{(0, 0)} \ov{\phi_{33}(\ov\chi,
    0)}.
\]
Write
\begin{equation}\begin{split}
\mathfrak{D}_1 = & - 4i\chi_1 \mu_{1j}\ov{e_{1,1j}} - 8 \chi_1
\chi_j  f_1^{(I_1+2I_n)},\ \ \mathfrak{D}_2 =  - 4i\chi_j
\mu_{11}\ov{e_{1,11}} - 8 \chi_1 \chi_j
f_1^{(I_1+2I_n)},\\
\mathfrak{D}_3 = & - 4i\chi_1  \mu_{1j} \ov{e_{2,1j}} - 8 \chi_1
\chi_j f_2^{(I_1+2I_n)},\ \ \mathfrak{D}_4 = -4i\chi_j
\mu_{11}\ov{e_{2,11}} - 8 \chi_1 \chi_j
 f_2^{(I_1+2I_n)}.
\end{split}\end{equation}
Here we used the fact that $f_1^{(I_j+2I_n)}=f_2^{(I_j+2I_n)}=0$ due
to Theorem 2.1. A direct computation shows that
$\mathcal{L}_1^2\mathcal{L}_j = \frac{\p^3}{\p z_1^2 \p z_j} +
2i\chi_j\frac{\p^3}{\p z_1^2 \p w} + 4i\chi_1\frac{\p^3}{\p z_1\p
z_j \p w} - 8 \chi_1 \chi_j\frac{\p^3}{\p z_1 \p w^2} - 4
\chi_1^2\frac{\p^3}{\p z_j \p w^2} - 8 i \chi_1^2 \chi_j\frac{\p^3}{
\p w^3}$. Thus
\begin{equation}
\begin{split}
& \mathcal{L}_1^2\mathcal{L}_j f_1|_{(0,0)} = 2i\chi_j\cdot 2
f_1^{(2I_1+I_n)} + 4i\chi_1 f_1^{(I_1+I_j+I_n)} - 8 \chi_1 \chi_j
\cdot 2 f_1^{(I_1+2I_n)}  = \mathfrak{D}_1 + \mathfrak{D}_2.
\end{split}
\end{equation}

\begin{equation}
\begin{split}
& \mathcal{L}_1^2\mathcal{L}_j f_2|_{(0,0)} = 2i\chi_j\cdot 2
f_2^{(2I_1+I_n)} + 4i\chi_1 f_2^{(I_1+I_j+I_n)}
 - 8 \chi_1 \chi_j f_2^{(I_1+2I_n)}  = \mathfrak{D}_3 + \mathfrak{D}_4.
\end{split}
\end{equation}

\medspace

As in (\ref{new form}), we have
\begin{equation}
\label{new formbb} \mathfrak{B}_1 \ov{\phi}(\chi,
0)^t=\mathfrak{A}_1.
\end{equation}
where
\begin{equation}
\mathfrak{B}_1:= \begin{pmatrix}
{\cal L}_j{\cal L}_k \phi_{hl} & {\cal L}_j{\cal L}_k \phi_{h\alpha} & {\cal L}_j{\cal L}_k \phi_{33} \\
{\cal L}_j{\cal L}_\beta \phi_{hl} & {\cal L}_j{\cal L}_\beta
\phi_{h\alpha} & {\cal L}_j{\cal L}_\beta
\phi_{33}\\
{\cal L}^2_1 {\cal L}_j\phi_{hl} & {\cal L}^2_1 {\cal L}_j
\phi_{h\alpha} &
{\cal L}^2_1 {\cal L}_j \phi_{33} \\
\end{pmatrix}_{1\leq j,h,l\leq 2, {3\leq \alpha,\beta\leq n-1}} \bigg|_{(0,0)}
=\begin{pmatrix}
B_{1, jk}\\
B_{1, j \beta}\\
\mathfrak{B}_{1, 33}
\end{pmatrix}.
\end{equation}
and $\mathfrak{A}_1=(A_{1,11}, A_{1,12}, A_{1,22}, A_{1,1\alpha},
A_{1,2\alpha}, \mathfrak{A}_{1,33})^t$ where
$\mathfrak{A}_{1,33}=-{\cal L}_1^2 {\cal L}_j f_1|_{(0,0)} \chi_1 -
{\cal L}_1^2 {\cal L}_j f_2|_{(0,0)}$ $\chi_2 =
-(\mathfrak{D}_1+\mathfrak{D}_2) \chi_1 - (\mathfrak{D}_3 +
\mathfrak{D}_4) \chi_2$, and $\mathfrak{B}_{1,33}=({\cal L}_1^2{\cal
L}_j  \phi_{hl}|_{(0,0)},\ {\cal L}_1^2{\cal L}_j
\phi_{h\alpha}|_{(0,0)},\ {\cal L}_1^2{\cal L}_j
\phi_{33}|_{(0,0)})$.

\medspace

We also have $\mathcal{L}_1^2\mathcal{L}_j\phi|_{(0,0)} =
2\phi^{(2I_1+I_j)}+2i\chi_j\cdot
2\phi^{(2I_1+I_n)}+4i\chi_1\phi^{(I_1+I_j+I_n)} - 8\chi_1
\chi_j\cdot 2\phi^{(I_1+2I_n)} - 4 \chi_1^2 \cdot
2\phi^{(I_j+2I_n)}$ so that

\begin{equation*}\begin{split}
& \mathcal{L}_1^2\mathcal{L}_j\phi_{11}|_{(0,0)} = 2\cdot 2i
\ov{e_{1,1j}} + 4i\chi_j \bigg(2i (\frac{\mu_{11}}{\mu_1}
\ov{e_{1,11}} e_{1,11}
+ \frac{\mu_{11}}{\mu_2} \ov{e_{2,11}} e_{2,11}) - \frac{2i}{\mu_{11}} f_1^{(I_1+2I_n)}\bigg)\\
& +4i\chi_1 \bigg(2i(\frac{\mu_{1j}}{\mu_1} \ov{e_{1,1j}}e_{1,11} +
\frac{\mu_{1j}}{\mu_2} \ov{e_{2,1j}}e_{2,11})\bigg)
 - 16 \chi_1 \chi_j\bigg(-\frac{2i}{\mu_1} e_{1,11} f_1^{(I_1+2I_n)}
- \frac{2i}{\mu_2} e_{2,11} f_2^{(I_1+2I_n)}\bigg)\\
& = -\frac{2i}{\mu_1}e_{1,11}(\mathfrak{D}_1 + \mathfrak{D}_2) -
\frac{2i}{\mu_2}e_{2,11}(\mathfrak{D}_3 + \mathfrak{D}_4) -
\frac{1}{\mu_{1j}\chi_1} \mathfrak{D}_1,
\end{split}\end{equation*}

\begin{equation*}\begin{split}
& \mathcal{L}_1^2\mathcal{L}_j\phi_{12}|_{(0,0)} =  2 \frac{2i
\sqrt{\mu_1}}{\sqrt{\mu_1+\mu_2}} \ov{e_{2,1j}} +4i\chi_j
\bigg(2i(\frac{\mu_{11}}{\mu_1} \ov{e_{1,11}} e_{1,12} +
\frac{\mu_{11}}{\mu_2}
\ov{e_{2,11}} e_{2,12}) -\frac{2i}{\mu_{12}} f_2^{(I_1+2I_n)} \bigg) \\
& + 4i\chi_1 \bigg(2i(\frac{\mu_{1j}}{\mu_1} \ov{e_{1,1j}} e_{1,12}
+ \frac{\mu_{1j}}{\mu_2} \ov{e_{2,1j}} e_{2,12}) \bigg)
- 16\chi_1 \chi_j \bigg(-\frac{2i}{\mu_1} f_1^{(I_1+2I_n)} e_{1,12} - \frac{2i}{\mu_2} f_2^{(I_1+2I_n)} e_{2,12}\bigg)\\
& = -\frac{2i}{{\mu_1}}e_{1,12}(\mathfrak{D}_1 + \mathfrak{D}_2)
-\frac{2i}{\mu_2}e_{2,12}(\mathfrak{D}_3 + \mathfrak{D}_4) -
\frac{1}{\mu_{12} \chi_1} \mathfrak{D}_3,
\end{split}\end{equation*}

\begin{equation*}
\begin{split}
& \mathcal{L}_1^2\mathcal{L}_j\phi_{22}|_{(0,0)} = 4i\chi_j
\bigg(2i(\frac{1}{\mu_1} \mu_{11} \ov{e_{1,11}}e_{1,22}
+ \frac{1}{\mu_2} \mu_{11} \ov{e_{2,11}}e_{2,22})\bigg) \\
& + 4i\chi_1\bigg(2i(\frac{1}{\mu_1} \mu_{1j} \ov{e_{1,1j}}e_{1,22}
+ \frac{1}{\mu_2} \mu_{1j} \ov{e_{2,1j}}e_{2,22})\bigg)\\
& - 16\chi_1 \chi_j \bigg(-\frac{2i}{\mu_1} f_1^{(I_1+2I_n)}
e_{1,22}
- \frac{2i}{\mu_2} f_2^{(I_1+2I_n)} e_{2,22}\bigg) \\
\end{split}
\end{equation*}
\begin{equation*}
\begin{split}
& = -\frac{2i}{\mu_1}e_{1,22}(\mathfrak{D}_1 + \mathfrak{D}_2) -
\frac{2i}{\mu_2}e_{2,22}(\mathfrak{D}_3 + \mathfrak{D}_4),
\end{split}
\end{equation*}

\begin{equation*}
\begin{split}
& \mathcal{L}_1^2\mathcal{L}_j\phi_{1k}|_{(0,0)} = 2 \cdot 2i
\delta_{jk}\ov{e_{1,11}} + 4i \chi_j \bigg(2i(\frac{1}{\mu_1}
\mu_{11} \ov{e_{1,11}}e_{1,1k}
+ \frac{1}{\mu_2} \mu_{11} \ov{e_{2,11}} e_{2,1k}) \bigg) \\
& + 4i\chi_1 \bigg(2i(\frac{1}{\mu_1} \mu_{1j} \ov{e_{1,1j}}e_{1,1k}
+ \frac{1}{\mu_2} \mu_{1j} \ov{e_{2,1j}} e_{2,1k})
- \frac{2i}{\sqrt{\mu_1}}\delta_{jk} f_1^{(I_1+2I_n)}\bigg)\\
& - 16 \chi_1 \chi_j \bigg(-\frac{2i}{\mu_1} f_1^{(I_1+2I_n)}
e_{1,1k}
- \frac{2i}{\mu_2} f_2^{(I_1+2I_n)} e_{2,1k}\bigg) \\
& = - \frac{2i}{\mu_1}e_{1,1k} (\mathfrak{D}_1 + \mathfrak{D}_2) -
\frac{2i}{\mu_2} e_{2,1k}(\mathfrak{D}_3 + \mathfrak{D}_4) -
\frac{1}{\sqrt{\mu_1} \chi_j} \delta_{jk}\mathfrak{D}_2,
\end{split}
\end{equation*}

\begin{equation*}
\begin{split}
& \mathcal{L}_1^2\mathcal{L}_j\phi_{2k}|_{(0,0)} = 2 \cdot \frac{ 2i
\sqrt{\mu_1}}{\sqrt{\mu_2}}\delta_{jk}\ov{e_{2,11}} + 4i \chi_j
\bigg(2i(\frac{1}{\mu_1} \mu_{11} \ov{e_{1,11}}e_{1,2k}
+ \frac{1}{\mu_2} \mu_{11} \ov{e_{2,11}} e_{2,2k}) \bigg) \\
& + 4i\chi_1 \bigg(2i(\frac{1}{\mu_1} \mu_{1j} \ov{e_{1,1j}}e_{1,2k}
+ \frac{1}{\mu_2} \mu_{1j} \ov{e_{2,1j}} e_{2,2k})
- \frac{2i}{\sqrt{\mu_2}} \delta_{jk}f_2^{(I_1+2I_n)}\bigg)\\
& - 16 \chi_1 \chi_j \bigg(-\frac{2i}{\mu_1} f_1^{(I_1+2I_n)}
e_{1,2k}
- \frac{2i}{\mu_2} f_2^{(I_1+2I_n)} e_{2,2k}\bigg)\\
& = - \frac{2i}{\mu_1}e_{1,2k} (\mathfrak{D}_1 + \mathfrak{D}_2) -
\frac{2i}{\mu_2} e_{2,2k}(\mathfrak{D}_3 + \mathfrak{D}_4)
 - \frac{1}{\sqrt{\mu_2} \chi_j} \delta_{jk}\mathfrak{D}_4,
\end{split}
\end{equation*}

\begin{equation*}
\begin{split}
& \mathcal{L}_1^2\mathcal{L}_j\phi_{33}|_{(0,0)} = 2 \cdot
\frac{2\mu_1}{\sqrt{\mu_2(\mu_1+\mu_2)}} \ov{e_{2,1j}}\\
& + 4i \chi_j \bigg(2i(\frac{1}{\mu_1} \mu_{11}
\ov{e_{1,11}}e_{1,33} + \frac{1}{\mu_2} \mu_{11} \ov{e_{2,11}}
e_{2,33}) - \frac{2 \mu_1}{\sqrt{\mu_1 \mu_2(\mu_1+\mu_2)}}
f_2^{(I_1+2I_n)}
\bigg) \\
& + 4i\chi_1 \bigg(2i(\frac{1}{\mu_1} \mu_{1j} \ov{e_{1,1j}}e_{1,33}
+ \frac{1}{\mu_2} \mu_{1j} \ov{e_{2,1j}} e_{2,33})\bigg)\\
& - 16 \chi_1 \chi_j \bigg(-\frac{2i}{\mu_1} f_1^{(I_1+2I_n)}
e_{1,33}
- \frac{2i}{\mu_2} f_2^{(I_1+2I_n)} e_{2,33}\bigg)\\
& = - \frac{2i}{\mu_1}e_{1,33} (\mathfrak{D}_1 + \mathfrak{D}_2) -
\frac{2i}{\mu_2} e_{2,33}(\mathfrak{D}_3 + \mathfrak{D}_4) + \frac{i
\sqrt{\mu_1}}{\sqrt{\mu_2(\mu_1+\mu_2)} \chi_1} \mathfrak{D}_3.
\end{split}
\end{equation*}

We write $\mathfrak{B}_2:=(Id - \mathfrak{G}_1) \mathfrak{B}_1$ and
$\mathfrak{A}_2:=(Id - \mathfrak{G}_1) \mathfrak{A}_1$ so that
\[
\mathfrak{B}_2 \ov\phi(\chi, 0)^t = \mathfrak{A}_2
\]
where
\[
\mathfrak{G}_1:=
\begin{pmatrix}
0 & 0 & 0 & \frac{2\chi_1}{\chi_3} & 0 & 0 & 0 & 0 \\
0 & 0 & 0 & \frac{\chi_2}{\chi_3} & 0 & \frac{\chi_1}{\chi_3} & 0 & 0 \\
0 & 0 & 0 & 0 & 0 & \frac{2\chi_2}{\chi_3} & 0 & 0 \\
0 & 0 & 0 & 0 & 0 & 0 & 0 & 0 \\
0 & 0 & 0 & \frac{\chi_\alpha}{\chi_3} & 0 & 0 & 0 & 0 \\
0 & 0 & 0 & 0 & 0 & 0 & 0 & 0 \\
0 & 0 & 0 & 0 & 0 & \frac{\chi_\alpha}{\chi_3} & 0 & 0 \\
0 & 0 & 0 & -\frac{\mathfrak{D}_1+\mathfrak{D}_2}{\mu_1 \chi_3} & 0
&
-\frac{\mathfrak{D}_3+\mathfrak{D}_4}{\mu_2 \chi_3} & 0 & 0 \\
\end{pmatrix}.
\]
By the construction of $\mathfrak{B}_2$ and $\mathfrak{A}_2$, we see
that (\ref{rel})(\ref{e1}) and (\ref{e2}) still hold. We further
calculate:
\begin{equation}
\begin{split}
&\mathfrak{B}_{2,33} = \mathfrak{B}_{1,33} + \frac{B_{1,13}}{\mu_1
\chi_3} (\mathfrak{D}_1 + \mathfrak{D}_2) +  \frac{B_{1,23}}{\mu_2
\chi_3} (\mathfrak{D}_3 + \mathfrak{D}_4) = \bigg(-\frac{1}{\mu_{1j}
\chi_1}, \mathfrak{D}_1\ - \frac{1}{\mu_{12}
    \chi_1} \mathfrak{D}_3,\ 0,\\
&\  - \frac{1}{\sqrt{\mu_1} \chi_3} \mathfrak{D}_2 +
\frac{\mathfrak{D}_1+\mathfrak{D}_2}{\sqrt{\mu_1} \chi_3}, 0, ...,
0,\ -\frac{1}{\sqrt{\mu_2} \chi_3} \mathfrak{D}_4 +
\frac{\mathfrak{D}_3+\mathfrak{D}_4}{\sqrt{\mu_2} \chi_3}, 0, ...,
0,\ \frac{i \sqrt{\mu_1}}{\sqrt{\mu_2(\mu_1+\mu_2)} \chi_1}
\mathfrak{D}_3 \bigg)
\end{split}
\end{equation}
and
\begin{equation*}\begin{split}
& \mathfrak{A}_{2,33} =
-\chi_1(\mathfrak{D}_1+\mathfrak{D}_2)-\chi_2(\mathfrak{D}_3+\mathfrak{D}_4)
+\frac{\mu_1\chi_1\chi_3}{\mu_1\chi_3}(\mathfrak{D}_1+\mathfrak{D}_2)
+\frac{\mu_2\chi_2\chi_3}{\mu_2\chi_3}(\mathfrak{D}_3+\mathfrak{D}_4)
= 0.
\end{split}\end{equation*}

\medspace

We turn to $\mathfrak{B}_{2,33} \ov\phi(\chi, 0)^t =
\mathfrak{A}_{2,33}$ to have
\begin{equation}\begin{split}
&  -\frac{\mathfrak{D}_1}{\sqrt{\mu_1}\chi_1} \ov{\phi_{1111}}
-\frac{\mathfrak{D}_3}{\mu_{12}\chi_1}\ov{\phi_{12}} -
\bigg(\frac{\mathfrak{D}_2}{\sqrt{\mu_1}\chi_3}
- \frac{\mathfrak{D}_1+\mathfrak{D}_2}{\sqrt{\mu_1} \chi_3}\bigg)\ov{\phi_{13}}\\
&\ \ - \bigg(\frac{\mathfrak{D}_4}{\sqrt{\mu_2}\chi_3} -
\frac{\mathfrak{D}_3+\mathfrak{D}_4}{\sqrt{\mu_2} \chi_3}
\bigg)\ov{\phi_{23}} + \frac{i \sqrt{\mu_1}
\mathfrak{D}_3}{\sqrt{\mu_2(\mu_1+\mu_2)}\chi_1}\ov{\phi_{33}}=0.
\end{split}\end{equation}
Substituting (\ref{rel}) to this equation, we obtain at $(\chi, 0)$
\begin{equation*}\begin{split}
& -\frac{\mathfrak{D}_1}{\sqrt{\mu_1}\chi_1}
\frac{\chi_1}{\chi_3}\ov{\phi_{13}}
-\frac{\mathfrak{D}_3}{\sqrt{\mu_1+\mu_2}\chi_2}
\bigg(\frac{\sqrt{\mu_1}}{\sqrt{\mu_1+\mu_2}} \frac{\chi_2}{\chi_3}
\ov{\phi_{13}} + \frac{\sqrt{\mu_2}}{\sqrt{\mu_1+\mu_2}}
\frac{\chi_1}{\chi_3} \ov{\phi_{23}}\bigg) \\
& - \bigg(\frac{\mathfrak{D}_2}{\sqrt{\mu_1}\chi_3} -
\frac{\mathfrak{D}_1+\mathfrak{D}_2}{\sqrt{\mu_1} \chi_3}\bigg)
\ov{\phi_{13}} - \bigg(\frac{\mathfrak{D}_4}{\sqrt{\mu_2}\chi_3} -
\frac{\mathfrak{D}_3+\mathfrak{D}_4}{\sqrt{\mu_2}
\chi_3}\bigg)\ov{\phi_{23}} + \frac{i \sqrt{\mu_1}
\mathfrak{D}_3}{\sqrt{\mu_2(\mu_1+\mu_2)} \chi_1} \ov{\phi_{33}} =
0.
\end{split}\end{equation*}

A quick simplification gives at $(\chi, 0)$
\begin{equation}\begin{split}
& - \frac{\mu_1 \mathfrak{D}_3 }{\sqrt{\mu_1}(\mu_1+\mu_2)
\chi_3}\ov{\phi_{13}} + \frac{\mu_1
\mathfrak{D}_3}{\sqrt{\mu_2}(\mu_1+\mu_2)\chi_3}\ov{\phi_{23}} +
\frac{i \sqrt{\mu_1} \mathfrak{D}_3}{\sqrt{\mu_2(\mu_1+\mu_2)}
\chi_2}\ov{\phi_{33}} =0.
\end{split}\end{equation}
Since $\phi^{(2I_1+I_j)}_{33}\neq 0$ in Case D, the polynomial
$\mu_1 \mathfrak{D}_3\neq 0$ because of (\ref{phi 1}). Hence we
obtain the same formula (\ref{33 eq}). As in the case A, we can
further get $deg(F)\le 3$. The proof for the case D$'$ is similar to
the case D.
\end{proof}

\medspace

Therefore the proofs of Proposition \ref{propdeg} and Theorem
\ref{mainthm} are complete.

\medskip

\bibliographystyle{amsalpha}

\bigskip \bigskip

\noindent Shanyu Ji (shanyuji@math.uh.edu), Department of
Mathematics, University of Houston, Houston, TX 77204, USA.\

\medskip

\noindent Wanke Yin (wankeyin@whu.edu.cn),  School of Mathematics
and Statistics, Wuhan University, Hubei 430072, China.

\end{document}